\documentclass{amsart}
\usepackage{amsmath,amssymb,mathrsfs,amsthm,amsfonts}
\usepackage[inline]{enumitem}
\usepackage[usenames,dvipsnames]{xcolor}
\usepackage{hyperref}
\usepackage[percent]{overpic}
\usepackage{comment}
\usepackage{algorithm}
\usepackage{algorithmic}
\usepackage{mathtools}
\usepackage{subcaption}
\usepackage{tikz}
\usepackage{bm}
\usetikzlibrary{calc}
\usetikzlibrary{arrows.meta,backgrounds}
\usepackage{multirow,array,longtable,booktabs}
\usetikzlibrary{arrows}

\hypersetup{
	colorlinks=true, linkcolor=blue,
	citecolor=ForestGreen
}
\usepackage[paper=letterpaper,margin=1in]{geometry}

\newtheorem{theorem}{Theorem}[section]
\newtheorem{lemma}[theorem]{Lemma}
\newtheorem{proposition}[theorem]{Proposition}

\theoremstyle{definition}

\newtheorem{definition}[theorem]{Definition}

\newtheorem{assumption}[theorem]{Assumption}
\newtheorem{remark}[theorem]{Remark}
\numberwithin{equation}{section}
\allowdisplaybreaks
\usepackage{acronym}

\acrodef{LDP}{Large Deviation Principle}

\newcommand{\E}{\mathbf{E}}	
\newcommand{\spt}{\mathrm{spt}} 

\newcommand{\be}{\begin{equation}}
\newcommand{\ee}{\end{equation}}
\newcommand{\bea} {\begin{array}{rl}}
\newcommand{\eea} {\end{array}}
\newcommand{\bepa}{\left\{ \begin{array}{l}}
\newcommand{\eepa} {\end{array}\right.}
\makeatletter

\newcommand\norm[1]{\left\Vert#1\right\Vert}

\newcommand{\R}{\mathbb{R}} 
\newcommand{\N}{\mathbb{N}}

\newcommand{\op}{\mathrm{op}}

\newcommand{\ds}{\displaystyle}

\newcommand{\cE}{\mathcal{E}}
\newcommand{\cC}{\mathcal{C}}

\newcommand{\cJ}{\mathcal{J}}

\renewcommand{\hat}{\widehat}

\usepackage{graphicx}

\newcommand{\bd}{{\bf d}}
\newcommand{\ov}{\overline}
\newcommand{\bx}{\bm{x}}

\newcommand{\bX}{\bm{X}}
\newcommand{\bY}{\bm{Y}}
\newcommand{\bP}{\mathbb{P}}
\newcommand{\cP}{\mathcal{P}}

\newcommand{\scrL}{\mathscr{L}}
\newcommand{\scrG}{\mathscr{G}}
\newcommand{\scrH}{\mathscr{H}}
\newcommand{\scrC}{\mathrm{C}}
\newcommand{\eps}{\varepsilon}

\newcommand{\linf}{L^{\infty}}

\newcommand{\tr}{\text{tr}}
\newcommand{\cA}{\mathcal{A}}
\newcommand{\cl}{\text{CL}}
\newcommand{\ol}{\text{OL}}

\newcommand{\dis}{\text{disp}}

\newcommand{\bv}{\bm{v}}

\newcommand{\bZ}{\bm{Z}}
\newcommand{\bA}{\bm{A}}
\newcommand{\cF}{\mathcal{F}}

\newcommand{\cK}{\mathcal{K}}
\newcommand{\cL}{\mathcal{L}}
\newcommand{\bbF}{\mathbb{F}}
\newcommand{\bu}{\bm{u}}

\newcommand{\mf}{\text{MF}}
\newcommand{\balpha}{\bm{\alpha}}
\newcommand{\ba}{\bm{a}}
\newcommand{\bp}{\bm{p}}
\newcommand{\diag}{D_{\text{diag}}}
\newcommand{\cS}{\mathcal{S}}
\newcommand{\cM}{\mathcal{M}}

\title[Displacement monotone MFGC]{Quantitative convergence for displacement monotone Mean Field Games of controls}

\author[J.\ Jackson]{Joe Jackson}
\address{J.\ Jackson,
	Department of Mathematics, University of Chicago,
	\newline\hphantom{\quad \ \ J. Jackson}
	5734 S.~University Avenue, Chicago, Illinois 60637 USA
}
\email{jsjackson@uchicago.edu}

\author[A.R. M\'esz\'aros]{Alp\'ar R. M\'esz\'aros}  
\date{\today}
\address{Department of Mathematical Sciences, University of Durham, Durham DH1 3LE, United Kingdom}
\email{alpar.r.meszaros@durham.ac.uk}

\begin{document}
	\begin{abstract}
In this paper we establish quantitative convergence results for both open and closed-loop Nash equilibria of $N$-player stochastic differential games in the setting of {Mean Field Games of Controls} (MFGC), a class of models where interactions among agents occur through both states and controls. Our analysis covers a general class of {\it non-separable Hamiltonians} satisfying a {\it displacement monotonicity} condition, along with mild regularity and growth conditions at infinity. A major novelty of our work is the rigorous treatment of a nontrivial fixed-point problem on a space of measures, which arises naturally in the setting of MFGC. Unlike prior works that either restrict to separable Hamiltonians -- rendering the fixed-point map trivial -- or assume convergence or regularity properties of the fixed point map, we develop a detailed structural analysis of this equation and its $N$-player analogue. This leads to new regularity results for the fixed-point maps and, in turn, to quantitative convergence of open-loop equilibria. We further derive sharp a priori estimates for the $N$-player Nash system, enabling us to control the discrepancy between open and closed-loop strategies, and thus to conclude the convergence of closed-loop equilibria. Our framework also accommodates common noise in a natural way.

	\end{abstract}
	
	
	\maketitle
	{
		}
    \setcounter{tocdepth}{1}
  \tableofcontents

\section{Introduction} 
 
This paper is dedicated to the proof of global in time quantitative convergence of both open-loop and closed-loop Nash equilibria associated with a general class of {\it Mean Field Games of Controls} (MFGC), in the framework of displacement monotone data. 

The theory of Mean Field Games (MFGs for short) was introduced roughly two decades ago by Lasry and Lions \cite{LasLio:06-I, LasLio:06-II, LasLio:07} and, independently, by Huang, Malham\'e and Caines \cite{HuaMalCai}. The aim of this theory is to characterize limits of Nash equilibria of symmetric stochastic differential games as the number of agents tends to infinity. Such $N$-player games are notoriously complex, and it is difficult to compute or analyze their Nash equilibria directly. In the mean field limit the curse of dimensionality can be lifted, and the equilibria for the limiting MFG can be described by the well-known ``MFG system": a coupled system of finite-dimensional PDEs consisting of a Kolmogorov--Fokker--Planck equation evolving forward in time and a Hamilton--Jacobi--Bellman equation evolving backwards in time.

While MFGs are in general easier to analyze than their $N$-player game counterparts, establishing a rigorous connection between the $N$-player games and their mean field limits is often challenging. The goal of this paper is to address this convergence problem in the setting of MFGC. In the remainder of the introduction, we first introduce the class of MFGC that we will study in the present paper, and state precisely our main results. We then discuss some general background on the convergence problem in MFG theory, and provide a comparison of our results with the (relatively sparse) existing literature on the convergence problem for MFGC. Finally, we give a brief overview of our proof techniques.

\subsection{Mean Field Games of Controls} In the vast majority of the literature on mean field games, the interaction between agents is modeled via their distributions in the state space. In applications, especially in economics or finance, agents often interact also through their controls. Such generalized MFGs are referred to in the literature as {\it Mean Field Games of Controls} (MFGC). Models of this type were introduced for the first time in the works \cite{GomVos:13, GomPatVos, GomVos:16} under the terminology of {\it extended mean field games}. For a probabilistic description of such models we refer to \cite{CarLac, ChaCriDel22, DjePosTan} and \cite[Section 4.6]{CD1}, and for an interesting application in finance, we mention \cite{CarLeh}. Such models have also been studied intensively using PDE or variational techniques, and for a non-exhaustive list of works we refer to \cite{Kob:22-cpde,Kob:22-nodea,AchKob, GraSir, SanShi, BonGiaPfe, GraMulPfe}.

The data for the MFGC considered here consist of functions 
\begin{align*}
    L : \R^d \times \R^d \times \cP_2(\R^{d} \times \R^d) \to \R, \quad G : \R^d \times \cP_2(\R^d) \to \R, 
\end{align*}
together with a non-negative constant $\sigma_0$. Here $\cP_2(\R^{d})$ and $\cP_2(\R^{d}\times\R^{d})$ stand for the set of Borel probability measures with finite second moment, supported on $\R^{d}$ and on $\R^{d}\times\R^{d}$, respectively. We will equip $\cP_2(\R^{d})$ and $\cP_2(\R^{d}\times\R^{d})$ with the classical 2-Wasserstein distance, denoted in both cases by $\bd_{2}$.

We refer to $L$ as the Lagrangian or running cost, and we refer to $G$ as the terminal cost. The constant $\sigma_0$ represents the intensity of the common noise. The typical input for $L$ is denoted by $(x,a,\mu)$, where $x$ stands for the position and $a$ for the control variable. For $\mu \in \cP_2(\R^d \times \R^d)$, we use the notation $\mu^x$ and $\mu^a$ for the first and second marginals, respectively. We define a Hamiltonian $H : \R^d \times \R^d \times \cP_2(\R^d \times \R^d) \to \R$ from the Lagrangian $L$ via the formula 
\begin{align*}
    H(x,p,\mu) := \sup_{a \in \R^d} \left\{ - a \cdot p  - L(x,a,\mu) \right\}.
\end{align*}
\noindent 
\textbf{The open-loop $N$-player game.} 
To describe the open-loop version of the game, we fix an initial measure $m \in \cP_2(\R^d)$, and work on the filtered probability space $\big(\Omega^N, \cF^N, \bbF^N =(\cF_t^N)_{0 \leq t \leq T}, \bP^N\big)$ described in Subsection \ref{subsec.probsetup}, which in particular hosts i.i.d. $\cF_0$-measurable random variables $(\xi^i)_{i = 1,\dots,N}$ and independent, $d$-dimensional $\bbF^N$-Brownian motions $(W^i)_{i = 0,\dots,N}$. We set
\begin{align*}
    \cA^{\ol}_N := \Big\{ \text{square-integrable, $\bbF^N$-progressive, $\R^d$-valued processes } \alpha = (\alpha_t)_{0 \leq t \leq T} \Big\}.
\end{align*}
Player $i$ chooses $\alpha^i \in \cA^{\ol}_N$, which determines the private state process $X^i$ via the dynamics
\begin{align} \label{dynamocs.openloop}
    dX_t^i = \alpha_t^i dt + \sqrt{2} dW_t^i + \sqrt{2\sigma_0} dW_t^0, \quad 0 \leq t \leq T, \quad X_0^i = \xi^i.
\end{align}
Player $i$ seeks to minimize a cost function of the form 
\begin{align} \label{costs.openloop}
    \cJ^{N,i}_{\ol}(\balpha) = \E\bigg[ \int_{0}^T \scrL^{N,i}\big(\bX_t,\balpha_t\big) dt + \scrG^{N,i}\big(\bX_T\big) \bigg], 
\end{align}
where for notational simplicity we define $\scrL^{N,i} : (\R^d)^N \times (\R^d)^N \to \R$ and $\scrG^{N,i} : (\R^d)^N \to \R$ by 
\begin{align}
    \scrL^{N,i}\big(\bx,\ba\big) := L\big(x^i,a^i,m_{\bx,\ba}^{N,-i}\big), \quad \scrG^{N,i}(\bx) := G\big(x^i, m_{\bx}^{N,-i}\big).
\end{align}
Here and throughout the paper, we use the notation
\begin{align*}
    m_{\bx}^{N,-i} := \frac{1}{N-1} \sum_{j \neq i} \delta_{x^j} \in \cP_2(\R^d), \quad m_{\bx,\ba}^{N,-i} := \frac{1}{N-1} \sum_{j \neq i} \delta_{(x^j,a^j)} \in \cP_2(\R^d \times \R^d), \quad \bx,\ba \in (\R^d)^N.
\end{align*}
An {\emph{open-loop Nash equilibrium}} for this $N$-player game is an $N$-tuple of controls $\bm \alpha = (\alpha^1,\dots,\alpha^N) \in (\cA^{\ol}_N)^N$ such that for each $i = 1,\dots,N$ and each $\beta \in \cA^{\ol}_N$, we have
\begin{align*}
    \cJ^{N,i}_{\ol}(\balpha) \leq \cJ^{N,i}_{\ol}\big(\balpha^{-i}, \beta \big), 
\end{align*}
where we use the notation 
\begin{align*}
    (\balpha^{-i}, \beta) := (\alpha^1,\dots,\alpha^{i-1},\beta,\alpha^{i+1},\dots,\alpha^N).
\end{align*}
It is typical to describe open-loop Nash equilibria by applying the Pontryagin maximum principle to the optimization problem faced by each player, and thus derive a system of forward-backward SDEs (FBSDEs). In this case, because of the interaction through the controls, we need a form of the \textit{generalized Isaacs condition} to hold. In other words, we need to know that there exists a function $\ba^N = (a^{N,1},\dots,a^{N,N}) :  (\R^d)^N \times (\R^d)^N \to (\R^d)^N$
which satisfies
\begin{align} 
    a^{N,i}(\bx, \bp) = \text{argmax}_{a\in\R^{d}} \left\{ a \mapsto  - a \cdot p^i - L(x^i,a,m_{\bx, \ba^N(\bx,\bp)}^{N,-i})  \right\}, 
\end{align}
or equivalently 
\begin{align} \label{def.ani}
    a^{N,i}(\bx,\bp) =  - D_p H\left(x^i,p^i, m_{\bx, \ba^N(\bx,\bp)}^{N,-i}\right).
\end{align}
Suppose that \eqref{def.ani} uniquely defines a map $\ba^N$, and that this map is Lipschitz continuous. To facilitate our discussion of the $N$-player game, we find it useful to introduce the following notation: for each $N \in \N$ and $i = 1,\dots,N$, we define $\scrH^{N,i} : (\R^d)^N \times (\R^d)^N \to {\R}$ by
\begin{align} \label{def.HNi.GNi}
      \scrH^{N,i}(\bx,\bp) := H\big(x^i,p^i, m_{\bx,\ba^N(\bx,\bp)}^{N,-i} \big).
\end{align}
We also define $\scrH^{N,i}_p, \scrH^{N,i}_x : (\R^d)^N \times (\R^d)^N \to {\R^d}$ by
\begin{align*}
   \scrH^{N,i}_p(\bx,\bp) := D_p H\Big(x^i,p^i,m_{\bx, \ba^N(\bx,\bp)}^{N,-i} \Big), \quad  \scrH^{N,i}_x(\bx,\bp) := D_x H\Big(x^i,p^i,m_{\bx, \ba^N(\bx,\bp)}^{N,-i} \Big).
\end{align*}
Then, under standard regularity assumptions on the data (in particular under Assumption \ref{assump.regularity} below), any open-loop equilibrium $\balpha = (\alpha^1,\dots,\alpha^N)$ must satisfy
\begin{align} \label{ol.pont.connection}
    \alpha_t^i = - \scrH^{N,i}_p\big(\bX_t^N,\bY_t^N\big) = a^{N,i}\big(\bX_t^N,\bY_t^N\big), 
\end{align}
where
\begin{align}  \label{XNYNZN}
    \big(\bX^N,\bY^N,\bZ^N\big) = \Big( \big(X^{N,i})_{i = 1,\dots,N}, \big(Y^{N,i})_{i = 1,\dots,N}, \big(Z^{N,i,j}\big)_{i = 1,\dots,N,j = 0,\dots,N} \Big)
\end{align}
is a solution to the FBSDE
\begin{align} \label{pontryaginNplayer}
    \begin{cases}
       \ds  dX_t^{N,i} = - \scrH^{N,i}_p\big(\bX_t^N,\bY_t^N\big) dt + \sqrt{2} dW_t^{i} + \sqrt{2\sigma_0} dW_t^0, \quad 0 \leq t \leq T,
       \vspace{.1cm} \\
      \ds   dY_t^{N,i} = \scrH^{N,i}_x\big( \bX_t^N,\bY_t^N \big) dt + \sum_{j = 0}^N Z_t^{N,i,j} dW_t^j, \quad 0 \leq t \leq T, 
       \vspace{.1cm} \\
     \ds    X_{0}^{N,i} = \xi^i, \quad Y_T^{N,i} = D_{x^i} \scrG^{N,i}\big(\bX_T\big).
    \end{cases}
\end{align}
We refer to \cite[Theorem 2.15]{CD1} for the game-theoretic stochastic maximum principle that we are invoking here.
By a solution to \eqref{pontryaginNplayer}, we mean a collection of processes of the form \eqref{XNYNZN}, with $X^{N,i}$ and $Y^{N,i}$ being $\R^d$-valued, continuous and $\bbF^N$-adapted, and $Z^{N,i,j}$ being $\R^{d\times d}$-valued, $\bbF^N$-progressively measurable, and square-integrable, such that the integral form of \eqref{pontryaginNplayer} is satisfied. The uniqueness of solutions to \eqref{pontryaginNplayer} for $N$ large enough under our standing assumptions is implied by Proposition \ref{prop.uniformstability} below. Existence can also be obtained by combining the stability estimate in Proposition \ref{prop.uniformstability} with the method of continuation, as in \cite{PengWumonotone}, in particular following directly the proof of Theorem 3.2 therein; we omit the details.
On the other hand, if 
\begin{align} \label{convexity}
    (x,a) \mapsto L(x,a, \mu) \text{ is convex for each } \mu, \quad x \mapsto G(x,m) \text{ is convex for each } m
\end{align}
then Pontryagin's principle also gives a sufficient condition (see e.g. \cite[Theorem 2.16]{CD1}), so that for any solution $(\bX,\bY,\bZ)$ to \eqref{pontryaginNplayer}, \eqref{ol.pont.connection} defines an open-loop Nash equilibrium. In particular, if Assumption \ref{assump.monotone} below holds with $\scrC_{L,x} = 0$ and $\scrC_{G} = 0$, then for $N$ large enough, there is a unique $N$-player open-loop Nash equilibrium, which takes the form \eqref{ol.pont.connection}.

\medskip
\noindent
\textbf{The closed-loop $N$-player game.} In the closed-loop version of the $N$-player game, we again fix an initial measure $m \in \cP_2(\R^d)$ and work on the same filtered probability space $\big(\Omega^N, \cF^N, \bbF^N =(\cF_t^N)_{0 \leq t \leq T}, \bP^N\big)$ described precisely in Subsection \ref{subsec.probsetup}. This time, player $i$ chooses a feedback $\alpha$ from the set of closed-loop controls 
\begin{align*}
    \cA^{\cl}_N = \Big\{ \alpha : [0,T] \times (\R^d)^N \to \R^d : \alpha \text{ is measurable and } |\alpha(t,\bx)| \lesssim 1 + |x| \Big\}.
\end{align*}
Player $i$ chooses $\alpha^i \in \cA^{\cl}_N$, and then $\balpha = (\alpha^1,\dots,\alpha^N)$ determines the state processes $\bX = (X^1,\dots,X^N)$ via the dynamics
\begin{align} \label{dynamics.closedloop}
    dX_t^i = \alpha^i(t,\bX_t) dt + \sqrt{2} dW_t^i + \sqrt{2\sigma_0} dW_t^0, \quad 0 \leq t \leq T, \quad X_{0}^i = \xi^i.
\end{align}
Player $i$ seeks to minimize the cost 
\begin{align} \label{costs.closedloop}
    \cJ_{\cl}^{N,i}(\balpha) = \E\bigg[ \int_{t_0}^T \scrL^{N,i}\big(\bX_t, \balpha(t,\bX_t)\big) dt + \scrG^{N,i}(\bX_T) \bigg], 
\end{align}
where $\bX$ is determined from $\balpha$ via the formula \eqref{dynamics.closedloop}.
A closed-loop Nash equlibirum is a tuple $\balpha = (\alpha^1,\dots,\alpha^N) \in (\cA^{\cl}_N)^N$ such that for each $i = 1,\dots,N$ and each $\beta \in \cA^{\cl}_N$, we have 
\begin{align*}
    \cJ^{N,i}_{\cl}\big(\balpha\big) \leq \cJ^{N,i}_{\cl}\big(\balpha^{-i},\beta\big).
\end{align*}
It is well-known that closed-loop Nash equilibria can be described by a system of PDEs typically known as the Nash system. To write this system in the MFGC case, we again need to know that the equation \eqref{def.ani} uniquely defines a map $\ba^N : (\R^d)^N \times (\R^d)^N \to (\R^d)^N$. If this is the case, then the Nash system takes the form 
\begin{align} \label{nashsystem}
    \begin{cases}
       \ds  - \partial_t u^{N,i} - \sum_{j = 1}^N \Delta_{j} u^{N,i} - \sigma_{0}\sum_{j,k = 1}^N \tr\big(D_{jk} u^{N,i} \big) + \scrH^{N,i}\big(\bx, \diag \bu^N\big)
       \\
       \ds \qquad + \sum_{j \neq i} \scrH^{N,j}_p \big(\bx, \diag \bu^N \big) \cdot D_j u^{N,i} = 0, 
      \quad  (t,\bx) \in [0,T] \times (\R^d)^N, 
       \\
       \ds u^{N,i}(T,\bx) = G^{N,i}(\bx), \quad \bx \in (\R^d)^N, 
    \end{cases}
\end{align}
where we again use the notation \eqref{def.HNi.GNi}, and where $\diag \bu^N = (D_1 u^{N,1},...,D_Nu^{N,N}) \in (\R^d)^N$. More precisely, if a sufficiently regular solution of \eqref{nashsystem} exists, then there is a unique closed-loop Nash equilibrium, which takes the form 
\begin{align} \label{cl.feedback}
    \alpha^i(t,\bx) = - \scrH^{N,i}_p\big( \bx, \diag \bu^N(t,\bx)\big), 
\end{align}
so that in particular the corresponding state process evolves according to 
\begin{align} \label{cl.dynamics}
    dX_t^{\cl,i} = - \scrH^{N,i}_p \big(\bX_t^{\cl}, \diag \bu^N(t,\bX_t^{\cl}) \big) dt + \sqrt{2} dW_t^i + \sqrt{2\sigma_0} dW_t^0, \quad 0 \leq t \leq T, \quad X_0^i = \xi^i.
\end{align}
\noindent
\textbf{The mean field game.} To formulate the mean field game, we again fix $m \in \cP_2(\R^d)$. We work on the filtered probability space $\big(\Omega, \cF, \bbF = (\cF_t)_{0 \leq t \leq T}, \bP\big)$ described explicitly in Subsection \ref{subsec.probsetup}, which in particular hosts an $\cF_0$-measurable random variable $\xi$ with law $m$ and independent $\bbF$-Brownian motions $W$ and $W^0$. We denote by $\bbF^0 = (\cF_t^0)_{0 \leq t \leq T}$ the augmented filtration generated by $W^0$. 

A mean field game is a fixed-point problem. In the setting of the MFGC considered here, we fix a continuous, $\cP_2(\R^d \times \R^d)$-valued, $\bbF^0$-adapted process $\mu = \big(\mu_t)_{0 \leq t \leq T}$. It is assumed that $\mu_0$ is deterministic, and $\mu_0^x = m$. A representative player chooses a square-integrable, $\bbF$-progressively measurable process $\alpha$ which determines their state process $X$ via the dynamics
\begin{align} \label{mfdynamics}
    dX_t = \alpha_t dt + \sqrt{2} dW_t + \sqrt{2\sigma_0} dW_t^0, \quad 0 \leq t \leq T, \quad X_0 = \xi.
\end{align}
The player's goal is to minimize the cost functional 
\begin{align} \label{mfcost}
    \cJ_{\mf}(\alpha) = \E\bigg[ \int_0^T L\big(X_t,\alpha_t,\mu_t\big) dt + G(X_T,\mu_T^x) \bigg].
\end{align}
A mean field equilibrium (MFE) is a measure flow $\mu$ such that for some minimizer $\alpha$ of the cost \eqref{mfcost}, the fixed-point relation
\begin{align} \label{fixedpoint}
    \cL^0(X_t,\alpha_t) = \mu_t, \quad \text{for a.e. $t$}
\end{align}
is satisfied.
In \eqref{fixedpoint}, $X$ denotes the corresponding state process determined by \eqref{mfdynamics}, and $\cL^0$ denotes the conditional law with respect to $\bbF^0$, i.e. $\cL^0(X_t,\alpha_t) = \cL(X_t, \alpha_t \, | \, \cF_T^0) = \cL^0X_t, \alpha_t \, | \, \cF_t^0)$. 

In the mean field setting, the fixed-point equation \eqref{def.ani} is replaced by the relation
\begin{align} \label{mf.fixedoint}
    \mu = \Big( \big(x,p\big) \mapsto \big(x, - D_pH(x,p,\mu) \big) \Big)_{\#} \nu.
\end{align}
In particular, to characterize the mean field equilibria, we need to know that for each fixed $\nu \in \cP_2(\R^d \times \R^d)$, there is a unique fixed point $\mu = \Phi(\nu)$ in \eqref{mf.fixedoint}. That is, we need a map $\Phi$ such that 
    \begin{align} \label{def.Phi}
    \Phi(\nu) = \bigg( \Big(x,p\big) \mapsto \big(x, - D_pH\big(x,p,\Phi(\nu)\big) \Big) \bigg)_{\#} \nu, \quad \text{for all } \nu \in \cP_2(\R^d \times \R^d).
\end{align}
If \eqref{mf.fixedoint} uniquely defines a Lipschitz mapping $\Phi : \cP_2(\R^d \times \R^d) \to \cP_2(\R^d \times \R^d)$, then any mean field equilibrium must satisfy $m_t = \cL^0(X_t)$, where $(X,Y,Z,Z^0)$ is a solution of the conditional McKean--Vlasov SDE
\begin{align} \label{pontryagin}
    \begin{cases}
       \ds  dX_t = - D_p \scrH\big(X_t,Y_t,\cL^0(X_t,Y_t) \big) dt + \sqrt{2} dW_t + \sqrt{2\sigma_0} dW_t^0, \quad 0 \leq t \leq T, 
       \vspace{.2cm}  \\
        \ds dY_t = D_x \scrH \big(X_t,Y_t, \cL^0(X_t,Y_t) \big) dt + Z_t dW_t + Z_t^0 dW_t^0, \quad 0 \leq t \leq T, \vspace{.2cm} 
        \\
        \ds X_0 = \xi, \quad Y_T = D_x G(X_T,\cL^0(X_T)), 
    \end{cases}
\end{align}
and we use the notation $\scrH:\R^d\times\R^d\times\mathcal{P}_2(\R^d\times\R^d)\to\R$,
\begin{align} \label{def.hatH}
    \scrH(x,p,\nu) := H\big(x,p,\Phi(\nu)\big).
\end{align}
By a solution of \eqref{pontryagin}, we mean a tuple $(X,Y,Z,Z^0)$ such that $X$ and $Y$ are continuous $\bbF$-adapted processes taking values in $\R^d$, $Z$ and $Z^0$ are square-integrable, $\bbF$-progressively measurable $\R^{d \times d}$-valued processes, and the integral form of \eqref{pontryagin} holds.
More precisely, if an MFE exists, then it must satisfy 
\begin{align} \label{mu.char}
    \mu_t = \cL^0\big( X_t, \alpha_t \big), \quad \alpha_t = - D_p \scrH\big(X_t,Y_t, \cL^0(X_t,Y_t)\big), 
\end{align}
for some solution $(X,Y,Z,Z^0)$ to \eqref{pontryagin}, see e.g. \cite[Theorem 3.27]{CD1}.
Under Assumptions \ref{assump.regularity} and \ref{assump.monotone}, \eqref{pontryagin} has a unique solution, see Proposition \ref{prop.app.wellposed}. Thus, there is at most one mean field equilibrium. Moreover, if the additional convexity condition \eqref{convexity} above holds, then Pontryagin's maximum principle becomes a sufficient condition, and so if $(X,Y,Z,Z^0)$ is any solution to \eqref{pontryagin}, then \eqref{mu.char} defines a Nash equilibrium.

\subsection{The main convergence result}

We now state our main assumptions.  We begin with the main regularity requirements on $L$ and $G$. 

\begin{assumption} \label{assump.regularity}
Both $L$ and $G$ are assumed to be fully $C^2$ (in the sense of \cite[Section 5.6.2]{CD1}), and all of their second derivatives are uniformly bounded. The first derivatives $D_{\mu} L$ and $D_{\mu} G$ are also uniformly bounded. Moreover, $L$ is uniformly strictly convex in $a$, i.e. there is a constant $C> 0$ such that 
\begin{align} \label{strictconvexity}
    D_{aa} L(x,a,\mu) \geq C I_{d \times d}, \,\, \text{for all } (x,a,\mu) \in \R^d \times \R^d \times \cP_2(\R^d \times \R^d),
\end{align}
and we also have the coercivity condition 
\begin{align} \label{coercive}
  \frac{1}{C} |a|^2 - C\big(1 + |x|^2 + \bd_2(\mu,\delta_0) \big) \leq  L(x,a,\mu) \leq C |a|^2 + C\big(1 + |x|^2 + \bd_2(\mu,\delta_0) \big).
\end{align}
Finally, the map
\begin{align*}
 \R^d \times \R^d \times \cP_2(\R^d \times \R^d) \times \R^d \times \R^d \ni  (x,a,\mu,x',a') \mapsto \frac{\delta}{\delta m} D_x L \big(x,a,\mu,x',a')
\end{align*}
is uniformly Lipschitz continuous, and likewise $\frac{\delta}{\delta m} D_a L$ and $\frac{\delta}{\delta m} D_x G$ are uniformly Lipschitz continuous. The Lipschitz continuity in the measure variable is with respect to $\bd_{1}$.
\end{assumption}

\begin{remark} \label{rmk.Ham}
    Under Assumption \ref{assump.regularity}, the Hamiltonian $H$ is $C^2$, with bounded second derivatives, and in addition $D_{\mu} H$ is uniformly bounded. Furthermore, there exists a constant $C'>0$ such that 
    \begin{align} \label{coerciveH}
         \frac{1}{C'} |p|^2 - C'\big(1 + |x|^2 + \bd_2(\mu,\delta_0) \big) \leq  H(x,p,\mu) \leq C' |p|^2 + C'\big(1 + |x|^2 + \bd_2(\mu,\delta_0) \big).
    \end{align}
Indeed, the growth conditions in \eqref{coerciveH} are consequences of Fenchel duality and the growth conditions on $L$. For instance, one deduces 
\begin{align*}
H(x,p,\mu) &=  D_{p}H(x,p,\mu)\cdot p - L(x,-D_{p}H(x,p,\mu),\mu)\\
& \le  D_{p}H(x,p,\mu)\cdot p - \frac{1}{C} |D_{p}H(x,p,\mu)|^2 +  C\big(1 + |x|^2 + \bd_2(\mu,\delta_0) \big)\\
& \le \varepsilon |D_{p}H(x,p,\mu)|^{2} + \frac{1}{\varepsilon}|p|^{2} - \frac{1}{C} |D_{p}H(x,p,\mu)|^2 +  C\big(1 + |x|^2 + \bd_2(\mu,\delta_0) \big),
\end{align*}
and by choosing $\varepsilon>0$ sufficiently small, we obtain the second inequality in \eqref{coerciveH}, and the first one is obtained in a similar fashion. 

Furthermore, the uniform bounds on $D_{\mu}H$ are deduced from the uniform bounds on $D_{\mu}L$ via the envelope theorem, since
$$
D_{\mu}H(x,p,\mu) = - D_{\mu}L(x,-D_{p}H(x,p,\mu),\mu),\ \ \forall (x,p,\mu)\in \R^{d}\times\R^{d}\times\cP_{2}(\R^{d}\times\R^{d}).
$$
\end{remark}

The second assumption concerns the \textit{displacement (semi-)}monotonicity of $L$ and $G$.
\begin{assumption} \label{assump.monotone}
There are constants $\scrC_{L,a} > 0$, $\scrC_{L,x}, \scrC_{G}  \geq 0$ such that
\begin{align} \label{dispmonotone.L}
    &\E\Big[ \big( D_aL(X,\alpha, \cL(X,\alpha)) - D_aL(\ov{X},\ov{\alpha}, \cL(\ov{X},\ov{\alpha}))  \big) \cdot (\alpha - \ov{\alpha})
   \nonumber  \\
    &\quad + \big(  D_x L(X,\alpha, \cL(X,\alpha)) - D_x L(\ov{X},\ov{\alpha}, \cL(\ov{X},\ov{\alpha})) \big) \cdot (X - \ov{X}) \Big] \geq \scrC_{L,a} \E\big[ |\alpha - \ov{\alpha}|^2 \big] - \scrC_{L,x} \E\big[ |X- \ov{X}|^2 \big]
\end{align}
and 
\begin{align} \label{dispmonotone.G}
    \E\Big[ \big( D_x G(X, \cL(X)) - D_x G(\ov{X}, \cL(\ov{X})) \big) \cdot  (X - \ov{X}) \Big] \geq -\scrC_{G} \E\big[ |X- \ov{X}|^2 \big], 
\end{align}
for all square-integrable $\R^d$-valued random vectors $X,\ov{X}, \alpha, \ov{\alpha}$. Finally, we have 
\begin{align}\label{disp:const}
    \scrC_{\dis} \coloneqq \scrC_{L,a} - T \scrC_{G} - \frac{T^2}{2} \scrC_{L,x} > 0.
\end{align}
\end{assumption}

\begin{remark} \label{rmk.data}
   Suppose that $L$ and $G$ take the form 
   \begin{align*}
       L(x,a,\mu) = \frac{\kappa_1}{2} |x|^2 + \frac{\kappa_2}{2} |a|^2 + L_0(x,a,\mu), \quad G(x,m) = \frac{\kappa_3}{2} |x|^2 + G_0(x,m), 
   \end{align*}
   with 
   \begin{align*}
       L_0 : \R^d \times \R^d \times \cP_2(\R^d \times \R^d) \to \R, \quad G_0 : \R^d \times \cP_2(\R^d) \to \R 
   \end{align*}
   being $C^2$ with bounded first and second derivatives. Then, it is straightforward to check that Assumption \ref{assump.regularity} holds, and that Assumption $\ref{assump.monotone}$ holds if $\kappa_1,\kappa_2,\kappa_3 > 0$ are chosen large enough (compared to the {lower} bounds on the second derivatives of $L_0$ and $G_0$). For similar examples and computations to verify the assumptions, we refer to \cite[Remark 2.8]{MesMou}.

   For another class of very natural examples, we can take $G$ to be displacement monotone, and $L$ of the form 
   \begin{align*}
       L(x,a,\mu) = \frac{1}{2} |a|^2 + \int_{\R^d \times \R^d} W\big( x-x',a-a'\big) \mu(dx,da'), 
   \end{align*}
   where $W : \R^d \times \R^d \to \R$ is convex and even, with bounded first and second derivatives. 
   Then one can check that $L$ satisfies the conditions appearing in Assumptions and \ref{assump.regularity} and \ref{assump.monotone}.


\end{remark}

To state our main result precisely, we are going to define for each $N\in\mathbb{N}$
\begin{align*}
    \bX_t^{\cl,N} = \big(X_t^{\cl,N,i}\big)_{i = 1,\dots,N}, \quad \bX_t^{\ol,N} = \big(X_t^{\ol,N,i}\big)_{i = 1,\dots,N} \quad \bX_t^{\mf,N} = \big(X_t^{\mf,i}\big)_{i = 1,\dots,N}
\end{align*}
as well as 
\begin{align*}
    \balpha_t^{\cl,N} = \big(\alpha_t^{\cl,N,i}\big)_{i = 1,\dots,N}, \quad \balpha_t^{\ol,N} = \big(\alpha_t^{\ol,N,i}\big)_{i = 1,\dots,N}, \quad \balpha_t^{\mf,N} = \big(\alpha_t^{\mf,i}\big)_{i = 1,\dots,N}
\end{align*}
in the following way. First, we set 
\begin{align*}
    \bX_t^{\ol,N} = \bX_t^N, \quad \alpha_t^{\ol,N,i} = - \scrH^{N,i}_p \big(\bX_t^N,\bY_t^N\big), 
\end{align*}
where $\big(\bX^N,\bY^N,\bZ^N\big)$ denote the unique solution to the FBSDE \eqref{pontryaginNplayer}. Next, assuming that an admissible solution $\bu^N$ of the Nash system \eqref{nashsystem} exists, in the sense of Definition \ref{def.admissible}, we define $\bX_t^{\cl,N}$ via the dynamics \eqref{cl.dynamics}, and we define 
\begin{align*}
    \alpha^{\cl,i}_t = -  \scrH^{N,i}_p \Big( \bX_t^{\cl,N}, \diag \bu^N\big(t, \bX_t^{\cl,N}\big) \Big). 
\end{align*}
Finally, we denote by $\big(\bX_t^{\mf,i}, \alpha_t^{\mf,i}\big)_{i = 1,\dots,N}$ the conditionally i.i.d. copies of the mean field equilibrium state and action in the sense discussed in Subsection \ref{subsec.ciid}; in particular, $\big(X_t^{\mf,i},\alpha_t^{\mf,i} \big)_{i \in \N}$ are conditionally i.i.d. with common law $\cL^0\big(X_t^{\mf,i},\alpha_t^{\mf,i}\big) = \mu_t$\footnote{In this paper we work on canonical probability spaces for concreteness and to avoid any measure-theoretic issues related to conditioning on the filtration generated by $W^0$. This means that $\mu$ is defined on the canonical space $(\Omega, \cF, \bP)$ discussed in Subsection \ref{subsec.probsetup}, rather than the larger canonical space $\big(\Omega^N, \cF^N,\bP^N \big)$ on which $\bX^{\ol,N}\balpha^{\ol,N}$, $\bX^{\cl,N}$ and $\bX$ are defined. However, we can view $\mu$ as a measure flow defined in the larger probability space $\big(\Omega^N, \cF^N, \bP^N\big)$ in an obvious way, by identifying it as a functional of the common noise $W^0$, see Subsection \ref{subsec.ciid} for details.}.
Finally, we use the notation 
\begin{align*}
    \mu_t^{\ol,N} = m^N_{\bX_t^{\ol,N},\balpha_t^{\ol,N}}, \quad \mu_t^{\ol,N,x} = m^N_{\bX_t^{\ol,N}}, 
\end{align*}
and likewise for $\mu_t^{\cl,N}$ and $\mu_t^{\cl,N,x}$.

With this notation in place, our main convergence result is as follows.
\begin{theorem} \label{thm.main}
    Let Assumptions \ref{assump.regularity} and \ref{assump.monotone} hold, and assume moreover that $m \in \cP_p(\R^d)$ for some $p > 4$. Then, there is a constant $C>0$ (independent of $N$) such that for all $N$ large enough, the following estimates hold:
    \begin{itemize}
        \item We have
        \begin{align} \label{mainest1}
            \E\bigg[ \sup_{0 \leq t \leq T}  \big|\bX_t^{\ol,N} - \bX_t^{\mf,N}\big|^2 + \int_0^T \big|\balpha_t^{\ol,N} - \balpha_t^{\mf,N}\big|^2 dt \bigg]  \leq CN r_{d,p}(N), 
        \end{align}
        and, as a consequence,
    \begin{align} \label{mainest2}
      \sup_{0 \leq t \leq T} \E\left[    \bd_2^2\big( \mu_t^x, \mu_t^{\ol,N,x}  \big) \right] +  \E\bigg[\int_0^T \bd_2^2\big( \mu_t, \mu_t^{\ol,N} \big) dt \bigg] \leq C r_{d,p}(N),
    \end{align}
    where 
        \begin{align} \label{ratedef}
        r_{d,p}(N) = \begin{cases}
                 N^{-1/2} & d < 4,
                 \\
                 N^{-1/2} \log(1 + N) & d = 4,
                 \\
                 N^{-2/d} & d > 4.
        \end{cases}
    \end{align}
    \item In addition, if $\bu^N$ is any admissible solution to \eqref{nashsystem} and $\bX^{\cl,N}$ and $\balpha^{\cl,N}$ are as defined above, then \begin{align} \label{mainest3}
            \E\bigg[ \sup_{0 \leq t \leq T} \big|\bX_t^{\cl,N} - \bX_t^{\ol,N}\big|^2 + \int_0^T \big|\balpha_t^{\cl,N} - \balpha_t^{\ol,N}\big|^2 dt \bigg] \leq \frac{C}{N}, 
        \end{align}
        and, as a consequence, 
    \begin{align} \label{mainest4}
      \sup_{0 \leq t \leq T}   \E\left[   \bd_2^2\big( \mu_t^x, \mu_t^{\cl,N,x}  \big) \right] +  \E\bigg[\int_0^T \bd_2^2\big( \mu_t, \mu_t^{\cl,N} \big) dt \bigg] \leq C r_{d,p}(N)
    \end{align}
    \end{itemize}
\end{theorem}

\subsection{Context and related literature} Here we discuss the background necessary to place Theorem \ref{thm.main} in context. We first discuss the convergence problem for general MFGs (i.e., without interaction through the controls), and then review the (relatively sparse) literature on convergence for MFGC. 

\medskip

\noindent 
\textbf{The convergence problem for standard MFGs.} The convergence problem has long been viewed as a central issue in mean field game theory, and has been approached through a variety of techniques, which can be roughly categorized into (i) analytical approaches based on the \textit{master equation}, (ii) \textit{compactness methods},  and (iii) stochastic-analytic methods. 

The master equation is a nonlocal and nonlinear PDE of hyperbolic type, set in general on the state space $\R^d\times\cP_2(\R^d)$, which plays the role of the value function for a standard stochastic control problem, and in some sense encodes all the relevant information about the MFG. Following P.-L. Lions' idea presented in his lectures at Coll\`ege de France, the work \cite{CarDelLasLio} established for the first time the quantitative convergence of closed-loop Nash equilibria via the master equation, as a consequence of a well-posedness and regularity theory for the master equation (see also \cite{JakRut} in connection to these results). In particular, \cite{CarDelLasLio} provided a clear `recipe' for treating the convergence problem: if one can construct a smooth enough solution to the master equation, then its ``projections" nearly solve the $N$-player Nash system, and provided there is non-degenerate idiosyncratic noise, this leads to a quantitative convergence result. We note also that in some cases, even suitably defined weak solutions to the master equation can help to establish convergence rates; see, for instance \cite{MouZha24}. 

The existence of smooth enough classical solutions to the master equation can be established either under suitable smallness assumptions (for instance, short time horizon), see, for instance in \cite{GanSwi, Carmona2018, May, CarCirPor, AmbMes}, or globally, under additional {\it monotonicity} and structural conditions on the data; see e.g. \cite{ChaCriDel22, CarDelLasLio, Carmona2018, GanMes, GanMesMouZha, BanMesMou, BanMes:master}. We refer also to \cite{GraMes:23, GraMes:24} for recent developments on different kinds of monotonicity conditions for MFGs in connection to master equations. 

Another approach to quantitative convergence was developed in \cite{LauTan, PosTan, JacTan}. In particular, the approach in these works is to characterize the mean field equilibrium via a McKean--Vlasov forward-backward stochastic differential equation (FBSDE), characterize the $N$-player equilibrium by an $N$-dimensional FBSDE system, and then establish the convergence of the $N$-dimensional FBSDE system towards the McKean--Vlasov FBSDE by a synchronous coupling argument. In particular, \cite{LauTan, PosTan} executed this approach in the case of sufficiently small time horizon or sufficiently ``dissipative" drift, while \cite{JacTan} treated the D-monotone case. We note that while the stochastic-analytic approach is perhaps best suited to the open-loop convergence problem (in particular \cite{LauTan} and \cite{JacTan} treat only the open-loop case), the recent preprints \cite{CirRed, CirJacRed} show that under appropriate conditions it is possible to quantify the gap between open and closed-loop equilibria of the $N$-player games, and thus extend convergence results from the open-loop to the closed-loop case.

Outside of the monotone regimes, when the uniqueness of mean field Nash equilibria is missing, a qualitative answer to the convergence problem has been obtained via compactness arguments (\cite{Fis, Lac:16}). The corresponding question for closed-loop equilibria required more sophisticated arguments, but was settled in a breakthrough work by Lacker (see \cite{Lac:20}; we also refer to \cite{LacLeF} for the extensions of these results). There have been substantial extensions and refinements of this compactness approach in recent years (see e.g. \cite{Dje:23-esaim, Dje:23-aap, Dje:23-preprint}) but we do not comment further because our focus here is on quantitative results in monotone settings.

\medskip 

\noindent \textbf{The convergence problem for MFGC.}
The literature on the convergence problem for MFGC is relatively sparse. Qualitative convergence results were obtained in various settings by Djete in the series of works \cite{Dje:22, Dje:23-esaim, Dje:23-aap}. These results are obtained in remarkably general settings and are based on compactness arguments, without requiring the uniqueness of the corresponding mean field equilibria. However, all of these works require the separability condition 
\begin{align} \label{cond:sep}
    H(x,p,\mu) = H_0(x,p,\mu^x) - f(x,\mu), 
\end{align}
where $\mu^x$ is the first marginal of $\mu$. In this case, the fixed point equation \eqref{fixedpoint} reduces to the much simpler one
\begin{align*}
    \cL\big(X_t, - D_p H_{0}(X_t,Y_t, \cL(X_t) \big) = \mu_t,
\end{align*}
which in particular is not an implicit equation for $\mu$ anymore. This means that no analysis of the fixed-point equations \eqref{def.Phi} and \eqref{def.ani} is necessary. While some applications in mathematical finance feature separable Hamiltonians (see for instance in \cite{CarLeh} and \cite[Section 4.7.1]{CD1}),
 the separability assumption \eqref{cond:sep} on the Hamiltonian seems to be mostly technical, and is not satisfied in many relevant economic models; see for instance the Hamiltonians in \cite{GraSir, GraMat} and the discussions in \cite[Section 3.1]{GraMes:23}.

To the best of our knowledge, the quantitative convergence for MFGC has been studied in only two prior works: in \cite{LauTan} Lauri\`ere and Tangpi consider the convergence of open-loop equilibria, while in \cite{PosTan} Possama\"i and Tangpi consider the convergence of both open and closed-loop equilibria.
Both \cite{LauTan} and \cite{PosTan} require either {\it short time} horizon or sufficiently {\it `dissipative' properties} on the underlying drift terms (which after time rescaling could eventually be comparable to short time results). In particular, neither \cite{PosTan} nor \cite{LauTan} obtain global in time results under monotonicity conditions. In addition, \cite{LauTan} requires the separability \eqref{cond:sep}, and \cite{PosTan} imposes {\it implicit regularity assumptions} on the {\it composition} of the drift and the corresponding fixed point maps. These jointly imposed regularity assumptions are then verified only when \eqref{cond:sep} holds.

\medskip

Compared to these prior contributions, ours seems to be the first work in the literature which:
\begin{itemize}
\item does not require the separability condition \eqref{cond:sep}, or assume any regularity properties of the fixed-point maps $\Phi$ and $\ba^N$;
\item provides a global in time quantitative result for MFG of controls via $D$-monotonicity, rather than smallness conditions.
\end{itemize}
As a consequence of D-monotonicity, all required properties of the fixed point maps are obtained explicitly as part of our analysis.

\subsection{Proof strategy and novelties of our approach}

One of the main challenges presented by MFGC (as compared to standard MFGs) is the need to study the fixed point equations \eqref{def.ani} and \eqref{def.Phi}. To even make sense of the equations \eqref{pontryaginNplayer}, \eqref{nashsystem} and \eqref{pontryagin}, which describe $N$-player open-loop equilibria, $N$-player closed-loop equilibria, and mean field equilibria respectively, we need to know that the fixed point equations \eqref{def.ani} and \eqref{fixedpoint} are uniquely solvable, and that the resulting maps $\ba^N$ and $\Phi$ have sufficient regularity. For the convergence problem, there is another key difficulty: for MFGC, the Hamiltonians $\scrH^{N,i}$ defined by \eqref{def.HNi.GNi} which govern the $N$-player game are \textit{not} simply the discretizations of the Hamiltonian $\scrH$ appearing in \eqref{def.hatH}. As a consequence, in order to obtain convergence, we need to know that $\ba^N$ converges to $\Phi$ in some sense. 

In Lemmas \ref{lem.ani} and \ref{lem.Phi}, we show that under our standing assumptions, (i) \eqref{def.ani} and \eqref{fixedpoint} indeed uniquely define maps $\ba^N : (\R^d)^N \times (\R^d)^N \to (\R^d)^N$ and $\Phi : \cP_2\big(\R^d \times \R^d\big) \to \cP_2\big(\R^d \times \R^d\big)$, (ii) $\Phi$ and $\ba^N$ satisfy Lipschitz bounds, the bounds on $\ba^N$ being uniform in $N$ in an appropriate sense, and (iii) $\ba^N$ converges to $\Phi$, in the sense that 
\begin{align*}
    m_{\bx, \ba^N(\bx,\bp)}^N \approx \Phi\big( m_{\bx,\bp}^N\big), \quad \bx,\bp \in (\R^d)^N, \quad N \gg 1. 
\end{align*}
We emphasize that the only prior results we know of which establish regularity of $\Phi$ (see e.g. \cite{CD1, CarLeh}) yield only continuity with respect to an appropriate topology under a Lasry--Lions monotonicity condition; we are not aware of any prior Lipschitz bound, which is a key step in treating the convergence problem. 
Our proofs of Lemmas \ref{lem.ani} and \ref{lem.Phi} are based on systematically exploiting displacement monotonicity via coupling arguments. 

Once we establish the key properties of $\ba^N$ and $\Phi$, the convergence of open-loop equilibria is obtained via a synchronous coupling argument, similar to the one carried out in \cite{JacTan}. The key observation is that displacement monotonicity implies that the FBSDE \eqref{pontryaginNplayer} is ``monotone" in an appropriate sense. This is not obvious in light of the dependence of the coefficients of the FBSDE \eqref{pontryaginNplayer} on the implicitly defined map $\ba^N$. 

The convergence of closed-loop equilibria is much more subtle, and does not follow easily from the estimates we obtain on the fixed-point maps. We follow the general strategy of \cite{CirJacRed} (and the one in the earlier work \cite{CirRed}) for $N$-player games without interaction through the controls, and aim to estimate directly the gap between open-loop and closed-loop equilibria by first obtaining estimates on the Nash system \eqref{nashsystem} which are independent of $N$. The main new challenge, however, is that when differentiating the Nash system \eqref{nashsystem}, we obtain new terms involving the derivatives of the \textit{composition}
\begin{align*}
    \hat{\ba}^N(t,\bx) = \ba^N\big(\bx, \diag \bu^N(t,\bx) \big) : [0,T] \times (\R^d)^N \to (\R^d)^N.
\end{align*}
In particular, a key point in the proof is to obtain estimates on the derivatives of $\hat{\ba}^N$ which are independent of $N$. This step significantly complicates the stochastic-analytic arguments appearing in the proofs of Lemma \ref{lem.offdiagbound} and Proposition \ref{prop.uiivi}. This means in particular that even having access to Lipschitz continuous fixed point maps is not enough to complete the road map from earlier works, and new technical challenges have to be overcome. 


\subsection{Outline of the paper}

The structure of the rest of the paper is as follows. In the short Section \ref{sec:not} we have collected some notations and preliminary results. Section \ref{sec.fixedpoint} concerns the fine regularity properties of the fixed point maps $\Phi$ and $\ba^{N}$. In Section \ref{sec:four} we have derived some general uniform in $N$ stability results on the Pontryagin systems, which together with the properties of the fixed point operators are then used for the convergence open-loop equilibria in the following section. Section \ref{sec:six} contains uniform in $N$ estimates on the Nash systems, which are crucial in for the convergence of closed-loop Nash equilibria. We conclude the paper with an appendix, where explain how to adapt some existing monotonicity-based arguments for McKean-Vlasov SDEs to establish the well-posedness and basic properties of \eqref{pontryagin}. 

\medskip

\noindent {\bf Acknowledgements.} J.J. is supported by the NSF under Grant No. DMS2302703. A.R.M. has been supported by the EPSRC New Investigator Award ``Mean Field Games and Master equations'' under award no. EP/X020320/1.

\section{Notation and preliminaries}\label{sec:not}

\subsection{Notation}
For $p\ge 1$ by $\cP_{p}(\R^{d})$ we denote the set of Borel probability measures supported on $\R^{d}$ which have finite $p^{th}$-moments, i.e. $\displaystyle M_{p}(\mu):=\left(\int_{\R^{d}}|x|^{p}\mu(d x)\right)^{\frac1p}<+\infty$, for any $\mu\in \cP_{p}(\R^{d})$. In our analysis in $\cP_{2}(\R^{d})$ we make use of the classical Monge--Kantorovich--Wasserstein distances, that we denote by $\bd_1$ and $\bd_2$. In particular, for $p\ge 1$ and $\mu,\nu\in \cP_p(\R^d)$ we set
$$
\bd_p(\mu,\nu):=\inf\left\{\iint_{\R^d\times\R^d}|x-y|^p d\gamma(x,y):\ \ \gamma\in\cP_p(\R^d\times\R^d),\ (\pi^x)_\sharp\gamma = \mu, (\pi^y)_\sharp\gamma = \nu\right\}^{\frac1p},
$$
where $\pi^x, \pi^y:\R^{d} \times \R^d \to\R^d$ stand for the canonical projection operators. For a measure defined on a product space, i.e. if $\mu\in\cP(\R^{d}\times\R^d)$ and if a typical variable has the form $(x,y)\in\R^d\times\R^d$, we use the shorthand notations $\mu^x = (\pi^x)_\sharp\mu$ and $\mu^y=(\pi^y)_\sharp\mu$ for the first and second marginals of $\mu$, respectively.

For a function $F:\cP_{2}(\R^{d})\to\R$ that is differentiable at $\mu\in\cP_{2}(\R^{d})$, we denote by $D_{\mu}F(\mu,\cdot):\spt(\mu)\to\R^{d}$ is intrinsic Wasserstein derivative at $\mu$. We refer to \cite{AGS, CD1, GanTud:19} for further details on this. We say that $F:\cP_{2}(\R^{d})\to\R$ has a first variation or flat derivative at $\mu\in\cP_{2}(\R^{d})$ if there exists $\frac{\delta}{\delta m}F(\mu):\R^{d}\to\R$, a continuous function, such that the limit
$$
\lim_{t\downarrow 0}\frac{F(\mu+t(\nu-\mu))-F(\mu)}{t}=\int_{\R^{d}}\frac{\delta}{\delta m}F(\mu)(x) d(\nu-\mu)(x)
$$
exists and has this representation for all $\nu\in\cP_{2}(\R^{d}).$ If $\frac{\delta}{\delta m}F(\mu)$ is $C^{1}$, then $F$ is Wasserstein differentiable at $\mu$ and $D_{\mu}F(\mu,x) = D_{x}\frac{\delta}{\delta m}F(\mu)(x).$ Again, when considering product spaces and if $F:\cP_{2}(\R^{d}\times\R^d)\to\R$ is differentiable at $\mu\in\cP(\R^{d}\times\R^d)$ and if a typical variable has the form $(x,y)\in\R^d\times\R^d$, we use the shorthand notations $D_{\mu}^xF$ and $D_{\mu}^yF$ to refer to the first $d$ coordinates and the last $d$ coordinates of the Wasserstein derivative $D_\mu F$, respectively.

\subsection{Probabilistic set-up}  \label{subsec.probsetup}

Given a time horizon $T > 0$, an initial measure $m \in \cP_2(\R^d)$ and $N \in \N$, we work on the canonical probability space $\big(\Omega^N, \cF^N, \bP^N \big)$, where
\begin{align*}
    \Omega^N = (\R^d)^N \times \cC^{N+1}, \quad \cC \coloneqq C\big([0,T] ; \R^d \big), 
\end{align*}
while $\cF^N$ is the Borel $\sigma$-algebra on $\Omega^N$, and $\bP^N$ is the measure 
\begin{align*}
    \bP^N = m^{\otimes N} \otimes \mathbb{W}^{\otimes N+1},
\end{align*}
with $\mathbb{W}$ the Wiener measure on $\cC$. We write $(\bx, \bm{\omega}) = \big(x^1,\dots,x^N,\omega^0,\omega^1,\dots,\omega^N\big)$ for the general element of $\Omega^N$, and we define canonical random variables $(\xi^i)_{i = 1,\dots,N}$ and processes $(W^i)_{i = 0,\dots,N}$ by 
\begin{align*}
    \xi^i\big(\bx,\bm{\omega}\big) = x^i, \quad W_t^i\big(\bx,\bm{\omega}\big) = \omega^i(t). 
\end{align*}
We denote by $\bbF^N = \big(\cF^N_t\big)_{0 \leq t \leq T}$ the (augmentation of) the filtration generated by $(\xi^i)_{i = 1,\dots,N}$ and $(W^i)_{i = 0,\dots,N}$. Thus, $(\xi^i)_{i = 1,\dots,N}$ are i.i.d. $\cF_0^N$-measurable random variables with common law $m$, and $(W^i)_{i = 0,\dots,N}$ are independent $\bbF^N$-Brownian motions. Moreover, we denote by $\bbF^{0}$ the filtration generated by $W^0$. We write $(\Omega, \cF, \bbF, \bP) = \big(\Omega^1,\cF^1,\bbF^1, \bP^1\big)$, and in this case we write $\xi = \xi^1$, $W = W^1$.

Given a random variable $X$ defined on $\Omega^N$ or $\Omega$, we use the notation $\cL^0\big(X \, | \, \cF_t^0\big)$ for the law of $X$ conditional on $\cF_T^0$. If $X$ is $\cF_t$-measurable, this is equivalent to conditioning on $\cF_t^0$. We use implicitly in various places throughout the text the fact that if $X$ is a continuous adapted process such that 
\begin{align*}
    \E\left[ \sup_{0 \leq t \leq T} |X_t|^2 \right] < \infty, 
\end{align*}
then there is a version of $\big(\cL^0(X_t)\big)_{0 \leq t \leq T}$ which is $\bbF^0$-adapted and has continuous paths in $\cP_2(\R^d)$, see e.g. Lemma 2.5 in \cite{Carmona2018}.

\subsection{Conditionally i.i.d. copies} \label{subsec.ciid}
Our main result, Theorem \ref{thm.main}, deals with ``conditionally i.i.d. copies" of $(X,\alpha)$, where $\alpha$ is the optimal control and $X$ is the state process corresponding to a mean field equilibrium $\mu$. This follows a standard construction, which is outlined for example in Section 2.2 in \cite{JacTan}, but we briefly explain what we mean for the reader's convenience.

The objects $X$, $\alpha$ and $\mu$ are all defined on the space $\big(\Omega, \cF, \bP\big)$. Moreover, by necessity of the Pontryagin maximum principle, there exist processes $Y$, $Z$, and $Z^0$ such that $(X,Y,Z,Z^0)$ solves \eqref{pontryagin}. Following the discussion in Section 2.2 of \cite{JacTan}, for each $N \in \N$,  we can view $(X,Y,Z,Z^0)$ as functionals on $\Omega = \R^d \times \cC^2$, and define processes $(X^{\mf,i},Y^{\mf,i},Z^{\mf,i},Z^{\mf,i,0})$ on $\Omega^N$ by setting 
\begin{align*}
    X^{\mf,i}\big(\bx, \bm{\omega}\big) = X^{\mf,i}\big(x^1,\dots,x^N,\omega^0,\dots,\omega^N\big) = X\big(x^i,\omega^0,\omega^i\big), 
\end{align*}
 and likewise for $Y^{\mf,i},Z^{\mf,i},Z^{\mf,i,0}$. In particular, 
\begin{align*}
    \cL\Big( X^{\mf,i},Y^{\mf,i},Z^{\mf,i} ,Z^{\mf,i,0}, W^0 \Big) = \cL\Big( X^,Y,Z ,Z^0, W^0\Big), \quad i = 1,\dots,N, 
\end{align*}
so that for each $t$, $\cL^0\big(X_t^{\mf,i}, Y_t^{\mf,i}\big) = \cL^0\big(X_t^{\mf}, Y_t^{\mf}\big)$, in the sense that they agree as functions of $W^0_{[0,t]}$. Setting $$\alpha_t^{\mf,i} = - D_p \scrH\big(X_t^{\mf,i}, Y_t^{\mf,i}, \cL^0(X_t^{\mf,i},Y_t^{\mf,i} \big),$$ we see that for each $t$, $\big(X_t^{\mf,i},\alpha_t^{\mf,i}\big)$ are i.i.d. conditionally on $\cF_t^0$, with common law $\mu_t$, in the sense that
\begin{align*}
    \cL^0\big(X_t^{\mf,i},\alpha_t^{\mf,i}\big)(\bx, \bm{\omega}) = \cL^0\big(X_t^{\mf,i},\alpha_t^{\mf,i}\big)(x^1,\dots,x^N, \omega^0,\dots,\omega^N) = \mu_t(\omega^0). 
\end{align*}

\section{Analysis of the fixed-point maps} \label{sec.fixedpoint}

For the results in this section, we do not need the full strength of the displacement monotonicity condition in Assumption \ref{assump.monotone}. Instead, we only need to know that \eqref{dispmonotone.L} holds for some $\scrC_{L,a} > 0, \scrC_{L,x} \geq 0$. We thus record the following weaker version of Assumption \ref{assump.monotone}.

\begin{assumption} \label{assump.weakmonotone}
    There are constants $\scrC_{L,a} > 0$, $\scrC_{L,x} \geq 0$ such that \eqref{dispmonotone.L} holds for all square-integrable $\R^d$-valued random vectors $X, \ov{X}, \alpha, \ov{\alpha}$. 
\end{assumption}

We also find it convenient in this section to introduce the notation 
\begin{align}
    \scrL^{N,i}_x(\bx,\ba) := D_x L\big( x^i,a^i, m_{\bx,\ba}^{N,-i}\big), \quad \scrL^{N,i}_a(\bx,\ba) := D_a L\big( x^i,a^i, m_{\bx,\ba}^{N,-i}\big).
\end{align}
We now proceed with our analysis of the fixed point equations, starting with \eqref{def.Phi}. First, we note that \eqref{def.Phi} is equivalent to
\begin{align} \label{Phidef.RV}
    \Phi\big(\cL(X,Y) \big) = \cL\Big( X, -D_pH \big(X,Y,\Phi(\cL(X,Y)) \big) \Big)
\end{align}
for each pair of square-integrable random variables. Our next objective is to show that under our standing assumptions the equation \eqref{Phidef.RV} uniquely defines a Lipschitz continuous map $\Phi$.

\begin{lemma} \label{lem.sublin}
Let Assumption \ref{assump.regularity} hold. Then there is a constant $C>0$ such that for all $(x,p,\mu)\in\R^d\times\R^d\times\cP_2(\R^d\times\R^d)$ we have
\begin{align*}
    |D_p H(x,p,\mu)| \leq C\left(1 + |x| + |p| + \left(\int_{\R^d\times\R^{d}} |y|^{2} \mu(dy) \right)^{1/4} \right).
\end{align*}
\end{lemma}

\begin{proof}
    We start from the well-known formula 
    \begin{align*}
        H(x,p,\mu) = - L\big(x,-D_pH(x,p,\mu),\mu\big) + D_pH(x,p,\mu) \cdot p.
    \end{align*}
   From the coercivity conditions \eqref{coercive} and \eqref{coerciveH}, we deduce that 
   \begin{align*}
       |D_p H(x,p,\mu)|^2 &\leq C \Big(1 + |H(x,p,\mu)| + |D_pH(x,p,\mu)||p| + |x|^2 + \bd_2(\mu,\delta_0) \Big)
       \\
       &\leq C \Big( 1 + |p|^2 + |D_pH(x,p,\mu)||p| + |x|^2 + \bd_2(\mu,\delta_0) \Big),
   \end{align*}
   and then an application of Young's inequality completes the proof.
\end{proof}

\begin{lemma} \label{lem.schauder}
Let Assumption \ref{assump.regularity} hold. Then for any bounded random variables $X$ and $Y$, there exists a fixed point of the map 
\begin{align*}
    \mu \mapsto \cL\big(X, - D_pH(X,Y,\mu) \big).
\end{align*}
\end{lemma}

\begin{proof}
    Define $\Psi:\cP_2(\R^d\times\R^d) \to \cP_2(\R^d\times\R^d)$ by $\Psi(\mu) := \cL\big(X,- D_pH(X,Y,\mu) \big)$. Denote by $\cM = \cM(\R^d \times \R^d)$ the space of signed Radon measures $\mu$ on $\R^d \times \R^d$, equipped with the Kantorovich--Rubinstein metric
    \begin{align*}
      \bd_{\text{KR}}(\mu,\nu) :=  \| \mu - \nu \|_{\text{KR}} = \sup_{\phi \text{ is 1-Lipschitz}, \,\, |\phi| \leq 1} \int_{\R^{d}\times\R^{d}} \phi \, d(\mu - \nu).
    \end{align*}
    Now for $R > 0$, consider the compact, convex subset $\cK_R$ of $(\cM, d_{\text{KR}})$ consisting of probability measures $\mu \in \cP_2(\R^d \times \R^d)$ such that $\mu(B_R \times B_R) = 1$, $B_R$ denoting the ball of radius $R$ in $\R^d$ centred at 0. By Lemma \ref{lem.sublin}, and the boundedness of $X$ and $Y$, we have that if $\mu \in \cK_R$, then 
    \begin{align*}
        |D_p H(X,Y,\mu)| \leq C(1 + \|X\|_{\infty} + \|Y\|_{\infty} + R^{1/2})
    \end{align*}
    almost surely, which means that for $R$ large enough, $\Psi(\cK_R \cap \cP_2(\R^d \times \R^d)) \subset \cK_R$.
    
    We next claim that $\Psi : \cK_R \cap \cP_2(\R^d \times \R^d) \to \cK_R \cap \cP_2(\R^d \times \R^d)$ is continuous with respect to the metric $\bd_{\text{KR}}$. Indeed, it is straightforward to check that $\bd_{\text{KR}}$ is equivalent to $\bd_1$ on $\cK_R \cap \cP_2(\R^d \times \R^d)$ (with a constant depending on $R$), and if $\mu^i$, $i = 1,2$ are in $\cK_R \cap \cP_2(\R^d \times \R^d)$, then by coupling
    \begin{align*}
        \bd_1\big( \Psi(\mu^1),\Psi(\mu^2) \big) \leq \E\big[ \big|D_p H(X,Y,\mu^1) - D_p H(X,Y,\mu^2) \big| \big] \leq C \bd_1(\mu^1,\mu^2), 
    \end{align*}
    thanks to the Lipschitz continuity of $D_pH$ in $\mu$. 
    
    Since the restriction of $\Psi$ to $\cK_R \cap \cP_2(\R^d \times \R^d)$, i.e. $\Psi : \cK_R \cap \cP_2(\R^d \times \R^d) \to \cK_R \cap \cP_2(\R^d \times \R^d)$ is  continuous with respect to $\bd_{\text{KR}}$, we can apply Schauder's Fixed Point Theorem to conclude. 
\end{proof}

\begin{lemma} \label{lem.Phi}
    Let Assumptions \ref{assump.regularity} and \ref{assump.weakmonotone} hold. Then the formula \eqref{Phidef.RV}
    uniquely defines a map $\Phi: \cP_2(\R^d \times \R^d) \to \cP_2(\R^d \times \R^d)$, which is Lipschitz continuous in the sense that 
    \begin{align*}
        \bd_2 \big( \Phi(\nu^1), \Phi(\nu^2) \big) \leq C \bd_2(\nu^1,\nu^2), 
    \end{align*}
    for some constant $C$ and all $\nu^1,\nu^2 \in \cP_2(\R^d)$.
\end{lemma}

\begin{proof}
   Fix $\nu^1,\nu^2\in \cP_2(\R^d \times \R^d)$, and choose $(X^i,Y^i)_{i = 1,2}$ so that $\nu^i = \cL(X^i,Y^i)$, and 
    \begin{align*}
        \bd_2^2(\nu^1,\nu^2) = \E\big[ |X^1 - X^2|^2 + |Y^1 - Y^2|^2 \big].
    \end{align*}
    Suppose first that for some $R > 0$, we have $\nu^i(B_R \times B_R) = 1$, so that $(X^i,Y^i)$ are bounded random variables. Then Lemma \ref{lem.schauder} shows that we can find $(\mu^i)_{i = 1,2}$ satisfying
    \begin{align*}
        \mu^i = \cL \Big( X^i, - D_p H(X^i,Y^i,\mu^i) \Big).
    \end{align*}
    Now for notational simplicity, set $\alpha^i = - D_p H(X^i,Y^i,\mu^i)$, and set 
    \begin{align*}
        \Delta X = X^1 - X^2, \quad \Delta Y = Y^1 - Y^2, \quad \Delta \alpha = \alpha^1 - \alpha^2, 
    \end{align*}
    as well as 
    \begin{align*}
        \Delta D_a L = D_a L(X^1,\alpha^1,\mu^1) -  D_a L(X^2,\alpha^2,\mu^2)  , \quad \Delta D_x L = D_x L(X^1,\alpha^1,\mu^1) -  D_x L(X^2,\alpha^2,\mu^2).
    \end{align*}
    Then by Assumption \ref{assump.weakmonotone}, we find that
    \begin{align*}
        \E\big[ |\Delta \alpha|^2 \big] &\leq C\E\Big[ \Delta D_a L \cdot \Delta \alpha + \Delta D_x L \cdot \Delta X + \big|\Delta X \big|^2 \Big]
        \\
        &
        = C \E\Big[ -\Delta Y \cdot \Delta \alpha + \Delta D_x L \cdot \Delta X \Big] +  C \E\big[ |\Delta X|^2 \big]
        \\
        &\leq \frac{1}{2} \E\big[ |\Delta \alpha|^2 \big] + C \E\big[ |\Delta X|^2 + |\Delta Y|^2 \big],
    \end{align*}
    where we used the identity $D_a L(x, - D_p H(x,p,\mu),\mu) = -p$, the Lipschitz regularity of $D_x L$, and Young's inequality. 
    We deduce that
    \begin{align} \label{FP.apriori}
        \bd_2^2 \big( \mu^1,\mu^2 \big) \leq \E\big[ |X^1 - X^2|^2 + |\alpha^1 - \alpha^2|^2 \big] \leq C \E\big[ |X^1 - X^2|^2 + |Y^1 - Y^2|^2  \big].
    \end{align}
    This proves that $\Phi$ is well-defined (i.e. single valued, since $\nu^1=\nu^2$ implies directly that $\mu^1=\mu^2$) on the dense subset $\cP_{\infty}(\R^d \times \R^d) \subset \cP_2(\R^d \times \R^d)$ consisting of measures with bounded support, and satisfies the stated Lipschitz bound on this set. We can thus extend $\Phi$ uniquely to a Lipschitz map on all of $\cP_2(\R^d \times \R^d)$, which still satisfies the equation \eqref{Phidef.RV}. Finally, the above stability argument shows that uniqueness holds even for $\nu \notin \cP_{\infty}(\R^d \times \R^d)$. This completes the proof.
\end{proof}

\begin{lemma} \label{lem.dispmonotonefiniteN}
   Suppose that Assumptions \ref{assump.regularity} and \ref{assump.weakmonotone} hold. Then there is a constant $C>0$ independent of $N$ such that 
    \begin{align} \label{dispmonotone.L.N}
        &\sum_{i = 1}^N \Big( \scrL^{N,i}_x\big(\bx,\ba\big) - \scrL^{N,i}_x\big(\ov \bx, \ov \ba \big) \Big) \cdot \big(x^i - \ov x^i\big) + \sum_{i = 1}^N \Big( \scrL^{N,i}_a\big(\bx,\ba\big) - \scrL^{N,i}_a\big(\ov \bx, \ov \ba \big) \Big) \cdot \big(a^i - \ov a^i\big)
       \nonumber  \\
        &\qquad \qquad \geq \scrC_{L,a} |\ba - \ov \ba|^2 - \scrC_{L,x}  |\bx - \ov \bx|^2 - \frac{C}{N} \big( |\bx - \ov \bx|^2 + |\ba - \ov \ba|^2 \big),  
    \end{align}
    for all $\bx,\ov \bx,\ba, \ov \ba \in (\R^d)^N$. 
    Similarly, if \eqref{dispmonotone.G} holds, there is a constant $C$ independent of $N$ such that 
    \begin{align} \label{dispmonotone.G.N}
        \sum_{i = 1}^N \Big( D_{x^i} \scrG^{N,i}(\bx) - D_{x^i} \scrG^{N,i}(\ov \bx) \Big) \cdot (x^i - \ov{x}^i) \geq - \big(\scrC_{G} + \frac{C}{N}\big) |\bx - \ov \bx|^2, 
    \end{align}
    for all $\bx,\ov \bx \in (\R^d)^N$. 
\end{lemma}

\begin{proof}
   First, choose $\bx, \ba, \ov{\bx}, \ov{\ba} \in (\R^d)^N$. If we specialize \eqref{dispmonotone.L} to the case that the joint law of $X, \ov{X}, \alpha, \ov{\alpha}$ is 
    \begin{align} \label{jointlaw}
        \cL(X,\ov{X},\alpha, \ov{\alpha}) = m_{\bx,\ov{\bx}, \ba, \ov{\ba}}^N = \frac{1}{N} \sum_{i = 1}^N \delta_{(x^i,\ov{x}^i,a^i,\ov{a}^i)}, 
    \end{align}
    so that in particular $\cL(X, \alpha) = m_{\bx,\ba}^N$, we obtain
    \begin{align} \label{DLmon.comp}
      \nonumber  & \sum_{i = 1}^N \Big( \big(D_x L(x^i,a^i,m_{\bx,\ba}^N) - D_x L(\ov{x}^i,\ov{a}^i,m_{\ov{\bx}, \ov{\ba}}^{N}) \big) \cdot (x^i - \ov{x}^i)
        \\
        &\qquad + \big(D_a L(x^i,a^i,m_{\bx,\ba}^N) - D_a L(\ov{x}^i,\ov{a}^i,m_{\ov{\bx}, \ov{\ba}}^{N}) \big) \cdot (a^i - \ov{a}^i) \Big) \geq \scrC_{L,a} |\ba - \ov \ba|^2 - \scrC_{L,x} |\bx - \ov \bx|^2.
    \end{align}
    Recalling the definition of $\scrL^{N,i}$, we see that we need only replace $m_{\bx,\ba}^{N}$ by $m_{\bx,\ba}^{N,-i}$ in the above expression, and control the resulting error. To this end, we write 
    \begin{align*}
       D_x L(x^i,a^i,m_{\bx,\ba}^{N,-i}) &= D_xL(x^i,a^i,m_{\bx,\ba}^N) 
       \\
       &\qquad + \int_0^1 \int_{\R^d} \int_{\R^d} \frac{\delta}{\delta m} D_x L\big(x^i,a^i,m_t,y,z \big) d( m_{\bx,\ba}^{N,-i} - m_{\bx,\ba}^N)(y,z)
        \\
        &= D_xL\big(x^i,a^i,m_{\bx,\ba}^N\big)  + \frac{1}{N(N-1)} \int_0^1 \sum_{j \neq i} \frac{\delta}{\delta m} D_x L\big(x^i, a^i, m_t, x^j,a^j \big) dt 
        \\
        &\qquad - \frac{1}{N} \int_0^1 \frac{\delta}{\delta m} D_x L\big(x^i, a^i, [m_{\bx,\ba}^N,m_{\bx,\ba}^{N,-i}]_t, x^i \big),
    \end{align*}
    where $m_t = [m_{\bx,\ba}^N,m_{\bx,\ba}^{N,-i}]_t$, and we use the notation $[\mu,\nu]_t = (1-t)\mu + t\nu$. Using the Lipschitz continuity of $\frac{\delta}{\delta m} D_x L$, we find that
   \begin{align} \label{lmon.comp2}
       \Big|  D_xL(x^i,a^i,m_{\bx,\ba}^N) - D_x L(x^i,a^i,m_{\bx,\ba}^{N,-i}) &- \Big( D_xL(\ov{x}^i,\ov{a}^i,m_{\ov{\bx},\ov{\ba}}^N) - D_x L(\ov{x}^i,\ov{a}^i,m_{\ov{\bx},\ov{\ba}}^{N,-i}) \Big) \Big|
    \nonumber    \\
       &\leq \frac{C}{N} \Big(|x^i - \ov{x}^i| + |a^i - \ov{a}^i| + \frac{1}{\sqrt{N}}|\bx - \ov \bx| + \frac{1}{\sqrt{N}} |\ba - \ov \ba| \Big).
   \end{align}   
    Of course, an analogous bound holds with $D_aL$ replacing $D_xL$. Adding and subtracting in \eqref{DLmon.comp}, and using \eqref{lmon.comp2} to control the resulting error gives \eqref{dispmonotone.L.N}. The proof of \eqref{dispmonotone.G.N} is similar.
\end{proof}

\begin{lemma} \label{lem.ani}
    Let Assumptions \ref{assump.regularity} and \ref{assump.weakmonotone} hold. Then there exists $N_0 \in \N$ such that for all $N \geq N_0$, the formula \eqref{def.ani}
    uniquely defines a map $\ba^N : (\R^d)^N \times (\R^d)^N \to (\R^d)^N$, and we have 
    \begin{align} \label{an.Lip}
        |\ba^{N}(\bx,\bp) - \ba^{N}(\ov{\bx},\ov{\bp})| \leq C  \big(|\bx - \ov \bx| + |\bp - \ov \bp| \big), 
    \end{align}
    and in addition 
    \begin{align} \label{ani.Lip}
        \left|a^{N,i}(\bx,\bp) - \ba^{N,i}(\ov{\bx},\ov{\bp})\right| \leq C \Big[ |x^i - \ov{x}^i| + |p^i - \ov{p}^i| + \frac{1}{\sqrt{N}} \big(|\bx - \ov \bx| + |\bp - \ov \bp|\big) \Big]
    \end{align}
    for some constant $C>0$ independent of $N$, and all $\bx, \ov \bx, \bp \ov \bp \in (\R^d)^N$. As a consequence, the derivatives $D_{x^j} a^{N,i}$, $D_{p^j}a^{N,i}$ exist a.e., and satisfy 
    \begin{align*}
        \Big\| |D_{x^i} a^{N,i}|^2 + |D_{p^i} a^{N,i}|^2 + N \sum_{j \neq i} \big( |D_{x^j} a^{N,i}|^2 + |D_{p^j} a^{N,i}|^2 \big) \Big\|_{\infty} \leq C, 
    \end{align*}
    for some constant $C$ independent of $N$ and all $N$ large enough.
\end{lemma}

\begin{proof}
    Fix $\bx, \bp \in (\R^d)^N$, and define a map $\Psi = (\Psi^1,\dots,\Psi^N) : (\R^d)^N \to (\R^d)^N$ via 
    \begin{align*}
        \Psi^i(\ba) := - D_p H(x^i,p^i, m_{\bx, \ba}^{N,-i}). 
    \end{align*}
    By Lemma \ref{lem.sublin}, there is a constant $C>0$ (which can depend on $\bx,\bp$, hence on $N$) such that
    \begin{align*}
        |\Psi^i(\ba)| \leq C + C \Big( \frac{1}{N} \sum_{j = 1}^N |a^j|^2 \Big)^{1/4} \leq C + C \max_{j = 1,\dots,N} |a^j|^{1/2}. 
    \end{align*}
    It follows that for large enough $R$, $\Psi(B_R^N) \subset B_R^N$, where $B_R$ is the ball of radius $R$ in $\R^d$ and $B_R^N$ is its $N$-fold product. Since $\Psi$ is continuous, we can apply Brouwer's Fixed Point Theorem to find a fixed point. Thus for each fixed $\bx, \bp \in (\R^d)^N$, there is at least one point $\ba \in (\R^d)^N$ satisfying $a^i = -D_pH\big(x^i,p^i,m_{\bx, \ba}^{N,-i}\big)$.

    Now suppose we are given $\bx, \bp, \ov{\bx}, \ov{\bp}\in (\R^{d})^{N}$, and that $\ba = (a^1,\dots,a^N)$ and $\ov{\ba} = (\ov{a}^1,\dots,\ov{a}^N)$ satisfy the equations
    \begin{align*}
        a^i = - D_p H\big(x^i,p^i, m_{\bx, \ba}^{N,-i}\big), \quad  \ov{a}^i = - D_p H\big(\ov{x}^i,\ov{p}^i, m_{\ov \bx, \ov{\ba}}^{N,-i}\big).
    \end{align*}
    Use Lemma \ref{lem.dispmonotonefiniteN} to find that for some $N_0 \in \N$ and all $N \geq N_0$, we have
    \begin{align*}
        \big| \ba - \ov \ba\big|^2 & \leq C \big|\bx - \ov \bx\big|^2  + {C}\sum_{i = 1}^N \Big[ \big(\scrL^{N,i}_x(\bx,\ba) - \scrL^{N,i}_x(\ov \bx, \ov \ba) \big) \cdot (x^i - \ov{x}^i) 
        \\
        &\qquad \qquad \qquad \qquad\qquad + \big(\scrL^{N,i}_a(\bx,\ba) -  \scrL^{N,i}_a(\ov \bx, \ov \ba) \big) \cdot (a^i - \ov{a}^i) \Big]
        \\
        &= C \big|\bx - \ov \bx\big|^2  + {C}\sum_{i = 1}^N \Big[ \big(\scrL^{N,i}_x(\bx,\ba) - \scrL^{N,i}_x(\ov \bx, \ov \ba) \big) \cdot (x^i - \ov{x}^i) \Big]
        - C \big(\bp - \ov \bp) \cdot (\ba - \ov \ba),
    \end{align*}
    where we have again used the identity $p = - D_a L(x,-D_pH(x,p,\mu),\mu)$.
To conclude with the desired estimates, let us observe that by the Lipschitz continuity of $D_xL$ (when this takes place with respect to $\bd_1$ in the measure component) and repeated use of Young's and Cauchy--Schwarz's inequalities, we have that there exists a constant $C>0$ independent of $N$, which might change from line to line, such that for all $\varepsilon>0$ we have
\begin{align*}
\sum_{i = 1}^N &\Big[ \big(\scrL^{N,i}_x(\bx,\ba) - \scrL^{N,i}_x(\ov \bx, \ov \ba) \big) \cdot (x^i - \ov{x}^i) \Big]\\
&\leq C\sum_{i = 1}^N \Big(|x^i-\ov{x}^i| + |a^i-\ov{a}^i| + \frac{1}{N-1}\sum_{j\neq i}\left(|x^j-\ov{x}^j| + |a^j-\ov{a}^j|\right) \Big)|x^i-\ov{x}^i|\\
&\leq C(1+1/\varepsilon ) \big|\bx - \ov \bx\big|^2 + \varepsilon \big|\ba - \ov \ba\big|^2.
\end{align*}
Handling similarly the term $\sum_{i = 1}^N \big(p^i - \ov{p}^i\big) \cdot (a^i - \ov{a}^i),$ by choosing $\varepsilon>0$ small enough, after rearranging the terms we find the bound \eqref{an.Lip}. 

Now we can infer \eqref{ani.Lip} from \eqref{def.ani} and \eqref{an.Lip}. Indeed, by the Lipschitz continuity of $D_pH$ (with respect to $\bd_1$ in the measure component) we have that there exists $C>0$ such that
    
\begin{align*}
 |a^i - \ov{a}^i| &\le \left|D_p H(x^i,p^i, m_{\bx, \ba^N(\bx,\bp)}^{N,-i}) - D_p H(\ov{x}^i,\ov{p}^i, m_{\bx, \ba^N(\ov\bx,\ov\bp)}^{N,-i}) \right|\\
 &\le C\Big(|x^i - \ov{x}^i| + |p^i-\ov{p}^i|+\frac{1}{N-1}\sum_{j\neq i}\left(|x^j-\ov{x}^j| + |a^j-\ov{a}^j|\right) \Big)
\end{align*}    
 We conclude by squaring the previous inequality, and using Cauchy--Schwarz and Young inequalities.   
    
\end{proof}

\begin{lemma}\label{lem:FPmap_disc} Let Assumptions \ref{assump.regularity} and \ref{assump.weakmonotone} hold.
    Let ${\Phi} : \cP_2(\R^d \times \R^d) \to \cP_2(\R^d \times \R^d)$ be defined by \eqref{def.Phi}. Let $N$ be large enough, and let $\ba^N : (\R^d)^N \times (\R^d)^N \to (\R^d)^N$ be defined by \eqref{def.ani}. Then there exists $N_0 \in \N$ and a constant $C$ independent of $N$ such that
    \begin{align*}
         \bd_2^{2} \Big( \Phi(m_{\bx,\bp}^N), {m}^N_{\bx, \ba^N(\bx,\bp)} \Big) \leq \frac{C}{N} + \frac{C}{N^2} \big(\big|\bx\big|^2 + \big|\bp\big|^2 \big),
    \end{align*}
    for all $N \geq N_0$ and all $\bx, \bp \in (\R^d)^N.$ 
\end{lemma}

\begin{proof}
    Let us start by fixing random variables $X$ and $Y$ such that $\cL(X,Y) = m_{\bx,\bp}^N$. In fact, it will be convenient to be more concrete here, setting $(X,Y) = \sum_{i = 1}^N (x^i,p^i) {\bf1}_{\Omega^i}$, where $\Omega^1,\dots,\Omega^N$ is a partition of $\Omega$ into sets of equal probability. Let us now set 
    \begin{align*}
        \alpha := - D_p\scrH\left(X,Y, m_{\bx,\bp}^N \right), \quad \ov{\alpha} := \sum_{i=1}^N a^{N,i}(\bx,\bp) {\bf1}_{\Omega^i}.
    \end{align*}
    Notice that 
    \begin{align*}
        \Phi(m_{\bx,\bp}^N) = \cL(X, \alpha), \quad m_{\bx, \ba^N(\bx,\bp)}^{N} = \cL(X, \ov{\alpha}).
    \end{align*}
    Moreover, using \eqref{dispmonotone.L} (since the $X$ component is the same) we have 
    \begin{align*}
        \E\big[|\alpha - \ov{\alpha}|^2 \big] &\leq C \E\Big[(\alpha - \ov{\alpha}) \cdot \big( D_a L(X,\alpha, \cL(X,\alpha)) - D_a L(X,\ov{\alpha}, \cL(X, \ov{\alpha}) ) \big) \Big]
        \\
        &= C\E\Big[ (\alpha - \ov{\alpha}) \cdot \big( D_a L(X, - D_p H(X,Y, \cL(X,\ov{\alpha})), \cL(X, \ov{\alpha})) 
        - D_a L(X, \ov{\alpha}, \cL(X,\ov{\alpha}))  \big) \Big]
        \\
        &\leq \frac{1}{2} \E\big[ |\alpha - \ov{\alpha}|^2 \big] + C \E\big[|-D_pH(X,Y,\cL(X,\ov{\alpha})) - \ov{\alpha}|^2]
        \\
        &{=} \frac{1}{2}\E\big[ |\alpha - \ov{\alpha}|^2 \big] + \frac{C}{N} \sum_{i = 1}^N \left|D_p H(x^i,p^i, \cL(X, \ov{\alpha})) - D_p H\left(x^i,p^i,m_{\bx,\ba^N(\bx,\bp)}^{N,-i}\right)\right|^2
        \\
        &= \frac{1}{2}\E\big[ |\alpha - \ov{\alpha}|^2 \big] + \frac{C}{N} \sum_{i = 1}^N \left|D_p H(x^i,p^i, m_{\bx,\ba^N(\bx,\bp)}^{N}) - D_p H\left(x^i,p^i,m_{\bx,\ba^N(\bx,\bp)}^{N,-i}\right)\right|^2
        \\
        &\leq  \frac{1}{2}\E\big[ |\alpha - \ov{\alpha}|^2 \big] + \frac{C}{N} \sum_{i = 1}^N \bd_2^2\big(m_{\bx, \ba^N(\bx,\bp)}^{N,-i}, m_{\bx, \ba^N(\bx,\bp)}^{N}\big)
        \\
        &\leq \frac{1}{2}\E\big[ |\alpha - \ov{\alpha}|^2 \big] + \frac{C}{N^2} \big|\ba^{N}(\bx,\bp)\big|^2 + \frac{C}{N^2} \big|\bx\big|^2,
    \end{align*}
 where in the second inequality we have used Young's inequality and Lipschitz continuity of $D_{a}L$ and in the penultimate inequality we have used the Lipschitz continuity of $D_{p}H$ in the measure component with respect to $\bd_{2}$.  
    Thus, 
    \begin{align} \label{alphabound}
       \bd_2^2 \Big( \Phi(m_{\bx,\bp}^N), {m}^N_{\bx, \ba^N(\bx,\bp)} \Big) \leq \E\big[ |\alpha - \ov{\alpha}|^2 ] \leq  \frac{C}{N^2}\big|\ba^N(\bx,\bp)\big|^2 {+ \frac{C}{N^2} \big|\bx\big|^2}.
    \end{align}
    To complete the proof, we need only show that there is a constant $C>0$ independent of $N$ such that for $N\in\N$ large enough
    \begin{align} \label{anlingrowth}
       \big|\ba^{N}(\bx,\bp)\big|^2 \leq CN +  C \big( |\bx|^2 + |\bp|^2 \big).
    \end{align}
    To this end, we use \eqref{alphabound} to estimate 
    \begin{align*}
        \big|\ba^{N}(\bm{0},\bm{0})\big|^2 \leq 2N M_2^2 \big(\Phi(m_{\bm{0},\bm{0}}^N)) + 2N \bd_2^2 \Big( \Phi(m_{\bm{0},\bm{0}}^N), {m}^N_{\bm{0}, \ba^N(\bm{0},\bm{0})} \Big)
      \leq CN + \frac{C}{N} \big|\ba^{N}(\bm{0},\bm{0})\big|^2 .
    \end{align*}
    For large enough $N$ we thus have $\frac{1}{N} \big|\ba^N(\bm{0}, \bm{0})\big|^2 \leq C.$
    Keeping in mind that $a^{N,i}(\bm{0},\bm{0}) = a^{N,j}(\bm{0},\bm{0})$ for $i \neq j$ by symmetry, we see that for $N\in\N$ large enough we have a bound on $|a^{N,i}(\bm{0}, \bm{0})|$ which is independent of $N$ and $i$. Together with the Lipschitz bound from Lemma \ref{lem.ani}, this is enough to obtain \eqref{anlingrowth}, which completes the proof.
\end{proof}

\section{Stability for the Pontryagin system and open-loop convergence}\label{sec:four}

We start by re-interpreting Lemma \ref{lem.dispmonotonefiniteN} in terms of the Hamiltonian.

\begin{lemma} \label{lem.HNmon}
    Let Assumptions \ref{assump.regularity} and \ref{assump.weakmonotone} hold. Then there is a constant $C$ which is independent of $N$ such the functions $\scrH^{N,i} : (\R^d)^N \times (\R^d)^N \to \R$ satisfy 
    \begin{align*}
        -\sum_{i = 1}^N &(x^i - \ov x^i) \cdot \big( \scrH^{N,i}_x(\bx,\bp) - \scrH^{N,i}_x(\ov \bx, \ov \bp) \big) + \sum_{i = 1}^N (p^i - \ov p^i) \cdot \big( \scrH^{N,i}_p(\bx,\bp) - \scrH^{N,i}_p(\ov \bx, \ov \bp) \big) 
        \\
        &\geq \scrC_{L,a} \sum_{i = 1}^N \big|\scrH^{N,i}_p(\bx,\bp) - \scrH^{N,i}_p(\ov \bx, \ov \bp) \big|^2 - \scrC_{L,x} |\bx - \ov \bx|^2 
        \\
        &\qquad \qquad - \frac{C}{N} \Big( \sum_{i = 1}^N \big|\scrH^{N,i}_p(\bx,\bp) - \scrH^{N,i}_p(\ov \bx, \ov \bp) \big|^2 + |\bx - \ov \bx|^2 \Big)
    \end{align*}
    for all $\bx,\ov \bx, \bp, \ov \bp \in (\R^d)^N$. 
\end{lemma}

\begin{proof}
    Given $\bx,\ov\bx,\bp,\ov \bp \in (\R^d)^N$, define $\ba, \ov \ba \in (\R^d)^N$ via $a^i = a^{N,i}(\bx,\bp)$, $\ov a^i = a^{N,i}(\ov \bx, \ov \bp)$. Then, we have the identities 
    \begin{align*}
        p^i = - D_a L\Big(x^i, - D_pH(x^i,p^i,m_{\bx,\ba}^{N,-i} \big), m_{\bx,\ba}^{N,-i} \Big) = - \scrL^{N,i}_a \big(\bx,\ba\big), \quad \scrH_x^{N,i}(\bx,\bp) = - \scrL_x^{N,i}(\bx,\ba), 
    \end{align*}
    and likewise for $\ov \bx$, $\ov \ba$, and $\ov \bp$. 
    Thus, the result follows from Lemma \ref{lem.dispmonotonefiniteN}.
\end{proof}

The following proposition shows that under displacement monotonicity, the $N$-player Pontryagin system enjoys certain dimension-free stability properties. To state it, we fix some $t_0 \in [0,T)$, and square-integrable $\cF^N_{t_0}$-measurable random variables $\zeta^1,\dots,\zeta^N$, which we do not require to be i.i.d. We aim to derive a stability estimate for the FBSDE
\begin{align} \label{pontryaginNplayer.gen}
    \begin{cases}
       \ds  dX_t^{i} = - \scrH^{N,i}_p\big(\bX_t,\bY_t\big) dt + \sqrt{2} dW_t^{i} + \sqrt{2\sigma_0} dW_t^0, \quad t_0 \leq t \leq T,
       \vspace{.1cm} \\
      \ds   dY_t^{i} = \scrH^{N,i}_x\big( \bX_t,\bY_t \big) dt + \sum_{j = 0}^N Z_t^{i,j} dW_t^j, \quad t_0 \leq t \leq T, 
       \vspace{.1cm} \\
     \ds    X_{t_0}^{i} = \zeta^i, \quad Y_T^{i} = D_{x^i} \scrG^{N,i}\big(\bX_T\big),
    \end{cases}
\end{align}
which is a slight generalization of \eqref{pontryaginNplayer}. In particular, we aim to compare the solution of \eqref{pontryaginNplayer.gen} to a process $\big(\hat \bX, \hat \bY,\hat \bZ\big)$ satisfying 
\begin{align} \label{pontryaginNplayer.gen.error}
    \begin{cases}
       \ds  d \hat X_t^{i} = \Big( - \scrH^{N,i}_p\big(\hat \bX_t,\hat \bY_t\big) + \mathcal{E}_t^{1,i}\Big) dt + \sqrt{2} dW_t^{i} + \sqrt{2\sigma_0} dW_t^0, \quad t_0 \leq t \leq T,
       \vspace{.1cm} \\
      \ds   d\hat Y_t^{i} = \Big( \scrH^{N,i}_x\big( \hat \bX_t,\hat \bY_t \big) + \mathcal{E}_t^{2,i} \Big) dt + \sum_{j = 0}^N \hat Z_t^{i,j} dW_t^j, \quad t_0 \leq t \leq T, 
       \vspace{.1cm} \\
     \ds    \hat X_{t_0}^{i} = \zeta^i + \mathcal{E}^{0,i}, \quad \hat Y_T^{i} = D_{x^i} \scrG^{N,i}\big(\hat \bX_T\big) + \mathcal{E}^{3,i},
    \end{cases}
\end{align}

\begin{proposition} \label{prop.uniformstability} Let Assumptions \ref{assump.regularity} and \ref{assump.monotone} hold. Then there are constants $N_0\in\N$ and $C>0$ such that for all $N \geq N_0$, the following holds. Suppose that we have a solution $(\bX,\bY,\bZ)$ to the Pontryagin system \eqref{pontryaginNplayer.gen}, as well as processes $(\hat{\bX}, \hat{\bY}, \hat{\bZ})$ satisfying \eqref{pontryaginNplayer.gen.error}
for some progressively measurable processes $\bm{\cE}^1 = (\cE^{1,i})_{i = 1,\dots,N}, \bm{\cE}^2 = (\cE^{2,i})_{i = 1,\dots,N}$, and square-integrable random variables $\bm{\cE}^{0} = (\cE^{0,i})_{i = 1,\dots,N} \in\cF_{t_0}^N$, $\bm{\cE}^{3} =( \cE^{3,i})_{i = 1,\dots,N} \in \cF_T^N$.
Then we have the bound 
\begin{align*}
    \E\bigg[ \sup_{t_0 \leq t \leq T} \big|\bX_t - \hat \bX_t \big|^2 + \int_{t_0}^T \big|\balpha_t - \hat \balpha_t\big|^2 dt \bigg] \leq C \E\bigg[ \big| \bm{\cE}^0 \big|^2 + \big| \bm{\cE}^3 \big|^2 + \int_{t_0}^T \Big( \big| \bm{\cE}_t^1 \big|^2 + \big| \bm{\cE}_t^2 \big|^2 \Big) dt\bigg]
\end{align*}
where $\balpha = (\alpha^1,\dots,\alpha^N)$ is given by
\begin{align*}
    \alpha_t^i = - \scrH^{N,i}_p \big(\bX_t,\bY_t\big), \quad \hat{\alpha}_t^i = - \scrH^{N,i}_p \big(\hat \bX_t, \hat \bY_t \big).
\end{align*}
In other words, \eqref{pontryaginNplayer.gen} satisfies a dimension-free $L^2$-stability estimate.
        
\end{proposition}

\begin{proof}
Choose $N_0$ large enough that the conclusions of Lemmas \ref{lem.ani} and \ref{lem.dispmonotonefiniteN} hold. For notational simplicity we set
    \begin{align*}
       \Delta X^i = X^{i} - \hat{X}^i, \quad \Delta Y^i = Y^{i} - \hat{Y}^i, \quad \Delta Z^{i,j} = Z^{i,j} - \hat{Z}^{i,j}, \quad \Delta \alpha_t^i = \alpha_t^i - \hat{\alpha}_t^i.
    \end{align*}
    We also denote by $\bm{\Delta \bX} = (\Delta X^1,\dots,\Delta X^N)$ and $\bm{\Delta \bY} = (\Delta Y^1,\dots,\Delta Y^N)$. We compute
\begin{align*}
    d \big( \bm{\Delta \bX_t}& \cdot \bm{\Delta \bY_t} \big)  = \bigg[\sum_{i = 1}^N \Delta X_t^i \cdot \Big( \scrH^{N,i}_x\big(\bX_t,\bY_t\big) - \scrH^{N,i}_x\big(\hat \bX_t, \hat \bY_t\big) \Big) 
    \\
    & - \sum_{i = 1}^N \Delta Y_t^i \cdot \Big( \scrH^{N,i}_p\big(\bX_t,\bY_t\big) - \scrH^{N,i}_p\big(\hat \bX_t, \hat \bY_t\big) \Big) 
    - \bm{\Delta Y}_t \cdot \bm{\cE}_t^1 - \bm{\Delta X}_t \cdot \bm{\cE_t^{2}} \bigg] dt + dM_t,
\end{align*}
where $M$ is a martingale whose form is not important in the following analysis. Integrate from $t_0$ to $T$, take expectations, and apply Lemma \ref{lem.HNmon} to obtain
\begin{align} \label{alpha.deltax.deltay}
 \nonumber    &\big(\scrC_{L,a}- \frac{C}{N}\big) \E\bigg[ \int_{t_0}^T |\bm{\Delta \alpha}_t|^2 dt \bigg]
    \\
    &\quad \leq  \E\bigg[  \big|\bm{\cE}^0\big| \big|\Delta \bY_{t_0} \big| + \big|\Delta \bX_T\big| \big|\bm{\cE}^{3}\big| + \frac{C}{N} \big|\bm{\Delta \bX}_T\big|^2 
   + \int_{t_0}^T  \Big(  \big|\bm{\Delta \bY}_t\big| \big| \bm{\cE}_t^{1}| +  \big|\bm{\Delta \bX}_t\big| \big| \bm{\cE}_t^{2}| + \frac{C}{N} \big|\bm{\Delta \bX}_t\big|^2 \Big) dt \bigg]
    \nonumber \\
    &\qquad \qquad + \E\bigg[ \scrC_{G} \big|\bm{\Delta \bX}_T\big|^2 + \int_{t_0}^T \scrC_{L,x} \big|\bm{\Delta X}_t|^2 dt \bigg].
\end{align}
Next, notice that \label{deltaxicomp}
\begin{align}
    \Delta X_t^i = -\cE_0^i + \int_{t_0}^t \Delta \alpha_s^i ds - \int_{t_0}^t \cE_s^{1,i} ds.
\end{align}
Thus, for any $\delta > 0$,  Young's inequality implies that there exists a constant $C>0$ such that
\begin{align*}
    |\Delta X_t^i|^2 \leq C\left(1+\frac{1}{\delta}\right) \Big(|\cE^{0,i}|^2 + (t-t_0)\int_{t_0}^t |\cE_s^{1,i}|^2 ds\Big) +  (t-t_0)(1 + \delta) \int_{t_0}^t |\Delta \alpha_s^i|^2 ds 
\end{align*}
from which we can easily deduce the bound 
\begin{align} \label{deltax}
    \E\left[|\Delta X^i_t|^2\right] \leq C\left(1+\frac{1}{\delta}\right) \E\left[\Big(|\cE^{0,i}|^2 + (t-t_0)\int_{t_0}^t |\cE_s^{1,i}|^2 ds\Big) \right] + (1 + \delta) (t - t_0) \E\left[\int_{t_0}^t |\Delta \alpha_s^i|^2 ds \right], 
\end{align}
and thus 
\begin{align} \label{deltaxint}
    \E\left[ \int_{t_0}^T |\Delta X^i_t|^2dt \right] \leq CT\left(1+\frac{1}{\delta}\right) \E\left[\Big(|\cE^{0,i}|^2 + T\int_{t_0}^{T} |\cE_s^{1,i}|^2 ds\Big) \right] + (1 + \delta)\frac{T^2}{2} \E\left[\int_{t_0}^{T} |\Delta \alpha_s^i|^2 ds \right].
\end{align}
Using Young's inequality repeatedly, and potentially increasing the constants $C>0$ (still independent of $N$) from one line to the other, we find that for all $\delta>0$ and $\varepsilon>0$, \eqref{alpha.deltax.deltay} implies
\begin{align} \label{alpha.deltax.deltay_2}
    &\left(\scrC_{L,a}- \frac{C}{N}\right) \E\bigg[ \int_{t_0}^T \big| \bm{\Delta \alpha}_t\big|^2 dt \bigg] \leq C \E\bigg[ \left( \frac{1}{\varepsilon}|\bm{\cE}^0|^{2}+  \varepsilon|\bm{\Delta Y}_{t_0}|^{2} +  \delta|\bm{\Delta X}_T|^{2} + \frac{1}{\delta}|\bm{\cE}^{3}|^{2}\right) + \frac{1}{N} \big|\bm{\Delta X}_T\big|^2 
   \nonumber \\
    & \qquad \qquad + \int_{t_0}^T  \Big( \varepsilon \big| \bm{\Delta Y}_t\big|^{2} + \frac{1}{\varepsilon} |\bm{\cE}_t^{1}|^{2} + \delta\big|\bm{\Delta X}_t\big|^{2} + \frac{1}{\delta}|\bm{\cE}_t^{2}|^{2} + \frac{1}{N} \big|\bm{\Delta X}_t\big|^2 \Big) dt \bigg]
    \nonumber \\
    &\qquad \qquad + \E\bigg[ \scrC_{G} \big|\bm{\Delta \bX}_T\big|^2 + \scrC_{L,x} \int_{t_0}^T  \big| \bm{\Delta X}_t \big|^2 dt \bigg].
\end{align}
Relying on \eqref{deltax} for the terms involving $|\Delta X^{i}_{T}|^{2}$ and on \eqref{deltaxint} for the terms involving $\ds\int_{t_{0}}^{T}|\Delta X^{i}_{t}|dt$, (recalling the value of $C_\dis$ from \eqref{disp:const}) the inequality \eqref{alpha.deltax.deltay_2} implies
\begin{align} \label{alpha.deltax.deltay2}
    \big(\scrC_{\dis} - \frac{C}{N} - \delta \big) \E\bigg[ \int_{t_0}^T \big| \bm{\Delta \alpha}_t \big|^2 dt \bigg]& \leq C_{\delta} \E\left[ (1+1/\varepsilon)|\bm{\cE}^0|^2 + \varepsilon \big| \bm{\Delta Y}_{t_0}\big|^{2} + (1/\delta)|\bm{\cE^{3}}|^{2} \right]\\
    & + C_{\delta}\E\left[\int_{t_0}^T  \Big( (1+1/\varepsilon) |\bm{\cE}_t^{1}|^2 + \varepsilon \big|\bm{\Delta Y}_t \big|^{2} + |\bm{\cE}_t^2|^{2} \Big) dt \right],
   \nonumber 
\end{align}
where $C_{\delta}>0$ is a constant depending on $\delta>0$ and $T>0$, but independent of $N$.

Choosing $\delta$ small enough and then $N$ large enough, we find that there is a constant $C$ and $N_0 \in \N$ such that for all $N \geq N_0$, and all $\eps > 0$,
\begin{align} \label{alpha.deltax.deltay3}
    \frac{\scrC_{\dis}}{2} \E\bigg[ \int_{t_0}^T \sum_{i = 1}^N |\Delta \alpha_t^i|^2 dt \bigg]& \leq \E\left[ (1+1/\varepsilon)|\bm{\cE}^0|^2 + \varepsilon \big| \bm{\Delta Y}_{t_0}\big|^{2} + (1/\delta)|\bm{\cE^{3}}|^{2} \right]\\
    & + C_{\delta}\E\left[\int_{t_0}^T  \Big( (1+1/\varepsilon) |\bm{\cE}_t^{1}|^2 + \varepsilon \big|\bm{\Delta Y}_t \big|^{2} + |\bm{\cE}_t^2|^{2} \Big) dt \right],
   \nonumber 
\end{align}

Now we can compute 
\begin{align*}
    d|\bm{\Delta Y}_t|^2 = 2 \sum_{i = 1}^N \Delta Y_t^i \cdot \Big( \scrH^{N,i}_x(\bX_t,\bY_t\big) - \scrH^{N,i}_x(\hat \bX_t, \hat \bY_t\big) - {\cE^{2,i}_{t}}\Big) dt + \sum_{i,j = 1}^N |\Delta Z_t^{i,j}|^2 dt + dM_t,  
\end{align*}
with $M$ being a martingale, whose particular form will not play a role in the analysis. Integrating in time and taking expectations, and using the Lipschitz regularity of the vector fields $(\scrH^{N,i}_x)_{i = 1,...,N}$ (which follows from Lemma \ref{lem.ani} and Assumption \ref{assump.regularity}) and of $(D_{x^i} \scrG^{N,i})_{i = 1,...,N}$, we find the bound
\begin{align*}
  \E\Big[ \big| \bm{\Delta \bY}_t\big|^2 \Big] &\leq C\E\left[ \big|\bm{\Delta X}_T\big|^2 + |\bm{\cE}^{3}|^2 + \int_t^T \Big(\big| \bm{\Delta \bY}_s\big|^2 + \big|\bm{\Delta \bX}_s\big|^2 + |\bm{\cE}_s^2|^2 \Big) dt \right],
\end{align*}
and then applying Gr\"onwall's inequality gives 
  \begin{align*}
  \E\Big[ \big| \bm{\Delta \bY}_t\big|^2 \Big] &\leq C\E\left[ \big|\bm{\Delta X}_T\big|^2 + |\bm{\cE}^{3}|^2 + \int_t^T \Big(\big|\bm{\Delta \bX}_t\big|^2 + |\bm{\cE}_t^2|^2 \Big) dt \right],
\end{align*}
Combining this with \eqref{deltax} and \eqref{deltaxint}, we find for all $t\in[t_{0},T]$
\begin{align}
 \label{deltay2}
   \E\Big[\big| \bm{\Delta Y}_t\big|^2 \Big] \leq  C_{\delta} \E\bigg[\Big(|\bm{\cE}^0|^2 + |\bm{\cE}^{3}|^2 \Big)  + \int_{t_0}^T  \Big(|\bm{\Delta \alpha}_s|^2 + |\bm{\cE}_s^{1}|^2 + |\bm{\cE}_s^{2}|^2 \Big) ds\bigg].
\end{align}
Combining \eqref{alpha.deltax.deltay3}, \eqref{deltax} and \eqref{deltay2}, and then using Young's inequality, we find that for any $\varepsilon > 0$, there is a constant $C_{\varepsilon} > 0$ such that for all $N\in\N$ large enough,
\begin{align*}
   \frac{\scrC_{\dis}}{2} \E\bigg[ \int_{t_{0}}^T  \big| \bm{\Delta \alpha}_t \big|^2 dt \bigg] &\leq 
    {\varepsilon} \E\bigg[ \int_{t_0}^T \big|\bm{\Delta \alpha}_t \big|^2 dt \bigg] 
     + C_{\varepsilon} \E\bigg[ \int_{t_0}^T \Big(  |\bm{\cE}_t^1|^2 + |\bm{\cE}_t^2|^2 \Big)dt + |\bm{\cE}^0|^2 + |\bm{\cE}^{3}|^2  \bigg]
\end{align*}
holds. 
Thus choosing $\varepsilon$ small enough, we see that for all $N\in\N$ large enough, we have 
\begin{align*}
   \E\bigg[ \int_{t_{0}}^T \big| \bm{\Delta \alpha}_t \big|^2 dt \bigg] \leq C\E\bigg[ \int_{t_0}^T \Big(  |\bm{\cE}_t^1|^2 + |\bm{\cE}_t^2|^2 \Big)dt + |\bm{\cE}^0|^2 + |\bm{\cE}^{3}|^2  \bigg]
\end{align*}
Then returning to \eqref{deltaxicomp} and \eqref{deltay2}, we get the desired bound.
\end{proof}

We now give aim to use Proposition \ref{prop.uniformstability} to prove the estimates \eqref{mainest1} and \eqref{mainest2} in Theorem \ref{thm.main}.
\begin{proof}[Proof of open-loop convergence]
    First, we note that by Pontryagin's maximum principle, we must have $\mu_t = \Phi\big( \cL(X_t,Y_t) \big)$,
    for some solution $(X,Y,Z,Z^0)$ to \eqref{pontryagin}. Moreover, by Proposition \ref{prop.app.wellposed}, \eqref{pontryagin} has a unique solution, and because $m_0 \in \cP_p$, Lemma \ref{lem.Lpbound} gives
    \begin{align} \label{momentbound}
      \E\left[ \sup_{0 \leq t \leq T} M_p^p\big(\mu_t\big) \right] < \infty. 
    \end{align}
    We recall that we have $\alpha_t^{\mf,i} = - D_p \scrH\big(X_t^{\mf,i}, Y_t^{\mf,i}, \nu_t\big)$, where in particular $\nu_t = \cL^0\big(X_t^{\mf,i},Y_t^{\mf,i}\big)$, and $(X_t^{\mf,i}, Y_t^{\mf,i}, Z_t^{\mf,i}, Z_t^{\mf,i,0})$ solves \eqref{pontryagin}. We also write $\bX^{\mf,N} = (X^{\mf,1},...,X^{\mf,N})$, and likewise for $\bY^{\mf,N}$, $\bZ^{\mf,N}$, $\bZ^{\mf,0,N}$. Adding and subtracting, we find
    \begin{align} \label{pontryagin.iidcopieserror}
    \begin{cases}
        dX_t^{\mf,i} = \Big( - \scrH^{N,i}_p \Big(\bX_t^{\mf,N}, \bY_t^{\mf,N} \Big) + \cE^{1,i}_t \Big)dt + \sqrt{2} dW_t^i + \sqrt{2\sigma_0} dW_t^0, 
       \vspace{.1cm} \\
        dX_t^{\mf,i} = \Big(\scrH^{N,i}_x \Big(\bX_t^{\mf,N}, \bY_t^{\mf,N}\Big) + \cE_t^{2,i} \Big) dt + Z_t^{\mf,i} dW_t^i + Z_t^{\mf,i,0} dW_t^0,
       \vspace{.1cm} \\
       X_0^{\mf,i} = \xi^i, \quad \hat{Y}_T^i = D_{x^i} \scrG^{N,i}\big(\bX_T^{\mf,i}\big) + \cE^{3,i},
    \end{cases}
\end{align}
where 
\begin{align*}
    &\cE_t^{1,i} := \scrH^{N,i}_p\big(\bX_t^{\mf,N}, \bY_t^{\mf,N}\big) - D_p \scrH\big(X_t^{\mf,i},Y_t^{\mf,i},\cL(X_t^{\mf,i},Y_t^{\mf,i})\big) 
    \\
    &\qquad \qquad = D_pH \Big( X_t^{\mf,i}, Y_t^{\mf,i}, m_{\bX_t^{\mf,N}, \ba^N(\bX_t^{\mf,N},\bY_t^{\mf,N})}^{N,-i} \Big) - D_p H\Big( X_t^{\mf,i}, Y_t^{\mf,i}, \Phi( \nu_t )\Big), 
    \\
   &\cE_t^{2,i} := \scrH^{N,i}_x\big(\bX_t^{\mf,N}, \bY_t^{\mf,N}\big) - D_x \scrH\big(X_t^{\mf,i},Y_t^{\mf,i},\cL(X_t^{\mf,i},Y_t^{\mf,i})\big) 
    \\
    &\qquad \qquad = D_xH \Big( X_t^{\mf,i}, Y_t^{\mf,i}, m_{\bX_t^{\mf,N}, \ba^N(\bX_t^{\mf,N},\bY_t^{\mf,N})}^{N,-i} \Big) - D_x H\Big( X_t^{\mf,i}, Y_t^{\mf,i}, \Phi( \nu_t )\Big),
    \\
    &\cE^{3,i} := D_xG (X_T^{\mf,i}, \nu_T^x) - D_x G\big(X_T^{\mf,i}, m_{\bX_T^{\mf,N}}^{N,-i}\big).
\end{align*}
where $\nu_t$ is the common conditional law $\nu_t = \cL^0\big(X_t^{\mf,i},Y_t^{\mf,i}\big)$, 
By tProposition \ref{prop.uniformstability}, we get 
 \begin{align*}
       \E\left[ \sup_{0 \leq t \leq T} \big|\bX_t^{\ol,N} - \bX_t^{\mf,N} \big|^2 + \int_0^T \big|\balpha_t^{\ol,N} - \balpha_t^{\mf,N}|^2 dt \right]
       \leq C \E\left[ \sum_{i = 1}^N |\cE^{3,i}|^2 + \int_0^T \sum_{i = 1}^N \Big( |\cE_t^{1,i}|^2 + |\cE_t^{2,i}|^2 \Big)dt \right].
    \end{align*}
It remains only to estimate the error terms. By the Lipschitz continuity of $D_p H$ and of $\Phi$
\begin{align*}
    |\cE_t^{1,i}|^2 &\leq C\bd_2^2\Big(\Phi(\nu_t), m_{\bX_t^{\mf,N}, \ba^N(\bX_t^{\mf,N}, \bY_t^{\mf,N})}^{N,-i} \Big)
    \\
    &\leq C\bigg( \bd_2^2\Big(\Phi(\nu_t), \Phi\big(m_{\bX_t^{\mf,N},\bY_t^{\mf,N}}^N \big)\Big) + \bd_2^2\Big(\Phi\big(m_{\bX_t^{\mf,N},\bY_t^{\mf,N}}^N \big), m_{\bX_t^{\mf,N}, \ba^N(\bX_t^{\mf,N}, \bY_t^{\mf,N})}^N \Big) 
    \\
    &\qquad \qquad \qquad + \bd_2^2\Big(m_{\bX_t^{\mf,N}, \ba^N(\bX_t^{\mf,N}, \bY_t^{\mf,N})}^N, m_{\bX_t^{\mf,N}, \ba^N(\bX_t^{\mf,N}, \bY_t^{\mf,N})}^{N,-i} \Big) \bigg).
\end{align*}
Applying the well-known result in \cite{FouGui} to bound the first term, Lemma \ref{lem:FPmap_disc} to control the first term, and the elementary bound $\bd_2^2\big(m_{\bx,\ba}^N,m_{\bx,\ba}^{N,-i}\big) \leq \frac{C}{N^2} \big(|\bx|^2 + |\ba|^2\big)$ for the third term, we arrive at
\begin{align*}
    \E\left[ \sum_{i = 1}^N |\cE_t^{1,i}|^2 \right] \leq CN \left( r_{d,p}(N) + \frac{1}{N} \right) \leq C_N r_{d,p}(N).
\end{align*}
Similar arguments give the same bound for $\cE^{2,i}$ and $\cE^{3,i}$,
which completes the proof of \eqref{mainest1}. Finally, \eqref{mainest2} follows from the triangle inequality and \eqref{mainest1}, because standard arguments based on the estimates for i.i.d. samples in \cite{FouGui} give 
 \begin{align} \label{mainest2'}
     \sup_{0 \leq t \leq T}   \E\left[   \bd_2^2\big( \mu_t^x, \frac{1}{N} \sum_{i = 1}^N \delta_{X_t^{\mf,i}}  \big) \right] +  \E\bigg[\int_0^T \bd_2^2\big( \mu_t, \frac{1}{N} \sum_{i = 1}^N \delta_{X_t^{\mf,i},\alpha_t^{\mf,i}} \big) dt \bigg] \leq C r_{d,p}(N).
    \end{align}
\end{proof}

\section{The $N$-player Nash system and closed-loop convergence}\label{sec:six}

The aim of this section is to prove the estimates \eqref{mainest3} and \eqref{mainest4}. To do this, we follow the roadmap of \cite{CirJacRed} to prove estimates on the Nash system \eqref{nashsystem}. A preliminary step is to obtain estimates on the decoupling field of the FBSDE \eqref{pontryaginNplayer}; this is the subject of the following subsection.

\subsection{Lipschitz estimates for the open-loop problem}

It turns out that the FBSDE \eqref{pontryaginNplayer} is ``decoupled" by the following PDE system: 
\begin{align} \label{pontryagin.pde}
\begin{cases} 
   \ds  - \partial_t v^{N,i} - \sum_{j = 1}^N \Delta_j v^{N,i} - \sigma_0 \sum_{j,k = 1}^N \tr\big( D_{x^jx^k} v^{N,i} \big)+ \sum_{j = 1}^N D_j v^{N,i} \scrH^{N,j}_p\big(\bx, \bv^N \big)
   \\
   \ds \qquad \qquad + \scrH^{N,i}_x\big(\bx, \bv^N \big) = 0, \quad  (t,\bx) \in [0,T] \times (\R^d)^N, 
    \vspace{.1cm} \\
   \ds  v^{N,i}(T,\bx) = D_{x^i} \scrG^{N,i}(\bx), \quad \bx \in (\R^d)^N.
\end{cases}
\end{align}
In \eqref{pontryagin.pde}, the unknown is $\bv^N = (v^{N,1},\dots,v^{N,N}) : [0,T] \times (\R^d)^N \to (\R^d)^N$. If the map $\ba^N$ is Lipschitz continuous (which holds for large enough $N$ by Lemma \ref{lem.ani}), then \eqref{pontryagin.pde} has a unique classical solution which is globally Lipschitz continuous in space on a small time horizon, and well-posedness can be extended to long times in light of the following a-priori estimate. This estimate uses the fact that if $\bv^N$ is a classical solution of \eqref{pontryagin.pde}, then by It\^o's formula, 
the unique solution $(\bX^N,\bY^N,\bZ^N)$, which satisfies
\begin{align*}
    Y_t^{N,i} = v^{N,i}(t,\bX_t), \quad Z_t^{N,i,j} = \sqrt{2} D_j v^{N,i}(t,\bX_t), \quad Z_t^{N,i,0} = \sqrt{2\sigma_0} \sum_{j = 1}^N D_jv^{N,i}(t,\bX_t), \quad i = 1,\dots,N.
\end{align*}

\begin{proposition} \label{prop.vLip}
    Let Assumptions \ref{assump.regularity} and \ref{assump.monotone} hold. Then there is a constant $C>0$ independent of $N$ {(possibly depending on the time horizon $T$)} such that for all $N\in\N$ large enough, the solution $\bv^N = (v^{N,1},\dots,v^{N,N})$ to \eqref{pontryagin.pde} satisfies 
    \begin{align*}
        \big|\bv^N(t,\bx) - \bv^N(t,\ov \bx)\big|^2 \leq C \big|\bx - \ov \bx\big|^2
    \end{align*}
    for each $t \in [0,T]$, $\bx,\ov \bx\in (\R^d)^N$. 
\end{proposition}

\begin{proof}
    Without loss of generality, we prove the estimate for $t_0 = 0$. Fix $\bx, \ov{\bx} \in (\R^d)^N$. Let $(\bX, \bY, \bZ)$ and $(\ov{\bX}, \ov{\bY}, \ov{\bZ})$ denote the solutions of \eqref{pontryagin} but started from the initial conditions $\bX_{0} = \bx$ and $\ov{\bX}_0 = \ov \bx$, respectively. By design, we have 
    \begin{align*}
        Y_0^i = v^i(0,\bx), \quad \ov{Y}_{0}^i = v^i(0, \ov \bx).
    \end{align*}
    In particular, using Proposition \ref{prop.uniformstability}, {since $\cE^{1,i}$, $\cE^{2,i}$ and $\cE^{3,i}$ are all zero,} we see that there is a constant $C$ such that for all $N\in\N$ large enough,
    \begin{align*}
        |\bm{v}^N(0,\bx) - \bm{v}^N(0,\ov{\bx})|^2 \leq C \big|\bx - \ov \bx\big|^2.
    \end{align*}
\end{proof}

\subsection{From the Pontryagin system to the Nash system}
Throughout this section, Assumptions \ref{assump.regularity} and \ref{assump.monotone} will be in force. We will always work with $N$ large enough that the conclusion of Lemma \ref{lem.ani} holds (i.e. the functions $\ba^N$ are well-defined and Lipschitz continuous). \begin{definition} \label{def.admissible}
    An admissible solution to \eqref{nashsystem} is a classical solution $\bu^N = (u^{N,1},\dots,u^{N,N})$ which is \textit{exchangeable}, in that 
    \begin{align*}
        u^{N,i}(t,\bx) = u^{N,1}\big(t, (x^i, \bx^{-i}) \big), \quad i = 1,\dots,N, \quad \bx \in (\R^d)^N, 
    \end{align*}
(here we use the notation $(x^i, \bx^{-i})=(x^{i},x^{1},\dots,x^{i-1},x^{i+1},\dots,x^{N})$) and in addition
    \begin{align*}
        u^{N,1}(t,\bx) = u^{N,1}(t,x^1, x^{\sigma(2)},\dots,x^{\sigma(N)}),
    \end{align*}
    for each permutation $\sigma$ of $\{2,\dots,N\}$, and one has bounded second spatial derivatives, i.e.
    \begin{align*}
        \| D_{kj} u^{N,i} \|_{\infty} < \infty, \quad \forall \,\, i,j,k = 1,\dots,N.
    \end{align*}
    An admissible closed-loop Nash equilibrium is a closed-loop Nash equilibrium which takes the form {\eqref{cl.feedback}} for some admissible solution $\bu^N= (u^{N,1},\dots,u^{N,N})$ to \eqref{nashsystem}.
\end{definition}
\begin{remark}
    The exchangeability condition in Definition \ref{def.admissible} is automatic if we have a uniqueness result for \eqref{nashsystem}, but again this is not immediate when $G$ and $H$ are allowed to have quadratic growth in $x$, so we make this part of the assumptions. 
\end{remark}
We will work throughout this section with an admissible solution $\bu^N$
to \eqref{nashsystem}, even if this is not specified every time.
Our goal is to obtain some uniform in $N$ information on these admissible solutions to the Nash system \eqref{nashsystem}. Our strategy, as in \cite{CirJacRed}, is going to be to argue that $\diag u^N = (D_1u^{N,1},\dots,D_Nu^{N,N})$ behaves like a small perturbation of the solution $\bv^N = (v^{N,1},\dots,v^{N,N})$ of \eqref{pontryagin.pde}. We first fix some notation. For simplicity, for $t_0 \in [0,T]$ and $\bx_0 \in (\R^d)^N$, we will denote by $\bX^{t_0,\bx_0} = (X^{t_0,\bx_0,1},\dots,X^{t_0,\bx_0,N})$ the closed-loop Nash equilibrium trajectory started from $(t_0,\bx_0)$, which satisfies
\begin{align*}
    dX_t^{t_0,\bx_0,i} &= - \scrH^{N,i}_p \Big(\bX_{t}^{t_0,\bx_0}, \diag \bu^{N}\big(t,\bX_t^{t_0,\bx_0}\big) \Big) dt + \sqrt{2} dW_t^i + \sqrt{2 \sigma_0} dW_t^0
    \\
    &= a^{N,i}\big(\bX_t^{t_0,\bx_0}, \diag \bu^N(t,\bX_t^{t_0,\bx_0}) \big)dt + \sqrt{2} dW_t^i + \sqrt{2 \sigma_0} dW_t^0, \quad t_0 \leq t \leq T, \quad X_{t_0}^{t_0,\bx_0,i} = x_0^i.
\end{align*} 
A special role will be played by the matrix $\bA^N = (A^{N,i,j})_{i,j = 1,\dots,N} \in (\R^{d \times d})^{N \times N} \simeq \R^{Nd \times Nd}$, given by the formula
\begin{align}\label{form:A_u}
    A^{N,i,j}(t,\bx) = D_{ji}u^{N,i}(t,\bx).
\end{align}
In particular, we are going to start by deriving a series of estimates on $(u^{N,1},\dots,u^{N,N})$ under the assumption that 
\begin{align} \label{startingpoint}
    \sup_{T_0 \leq t_0 \leq T, \bx_0 \in (\R^d)^N} \E\bigg[ \int_{t_0}^T |\bA^N(t, \bX_t^{t_0,\bx_0})|_{\op}^2 dt\bigg] \leq K, 
\end{align}
for some $T_0 \in [0,T)$ and $K > 0$, where $|\cdot |_{\op}$ denotes the operator norm. It will also be useful to use the notation 
\begin{align} \label{hatadef}
 \hat{\ba}^N = (\hat{a}^{N,1},\dots,\hat{a}^{N,N}), \quad   \hat{a}^{N,i}(t,\bx) = a^{N,i}\big(\bx, \diag \bu^N(t,\bx) \big).
\end{align}
We now obtain some bounds on the matrix $D \hat{\ba}^N = (D_j \hat{a}^{N,i})_{i,j = 1,\dots,N}$. 

\begin{lemma} \label{lem.hataprops}
   Let Assumptions \ref{assump.regularity} and \ref{assump.monotone} hold. Then there is a constant $C>0$ such that for all $N\in\N$ large enough, the bounds
    \begin{align} \label{Dhata.opbound}
        |D \hat{\ba}^{N}|_{\op} \leq C\big(1 +  |\bA^N|_{\op}\big)
    \end{align}
    and
    \begin{align} \label{Dhata.d2lip}
         \Big\| \sum_{j \neq i} |D_j \hat{a}^{N,i}|^2 \Big\|_{\infty} \leq C \Big( (1+|\bA^N|^2_{\op})/N + \sum_{j \neq i} |D_{ji} u^{N,i}|^2 + \frac{1}{N} \sum_{j \neq i} \sum_{k \neq j} |D_{jk} u^{N,k}|^2 \Big), \quad i = 1,\dots,N
    \end{align}
    holds almost everywhere on $[0,T] \times (\R^d)^N$.
\end{lemma}

\begin{proof}
    We argue as if $\hat{\ba}$ and $D_iu^{N,i}$ are $C^1$, the general case (in which they are just Lipschitz) follows from a mollification argument. The bound \eqref{Dhata.opbound} follows from teh fact 
    \begin{align*}
      (\R^d)^N \ni \bx \mapsto  \hat{\ba}^N(t,\bx) = \ba^N\big( \bx, \diag \bu^N(t,\bx) \big) \in (\R^d)^N
    \end{align*}
    as a composition of the maps
    \begin{align*}
       \bx \mapsto \big(\bx, \diag \bu^N(t,\bx) \big), \quad (\bx, \bp) \mapsto \ba^N\big(\bx, \bp),
    \end{align*}
    whose Lipschitz constants are bounded by $C(1 \|\bA^N\|_{\op})$ and $C$, respectively. 
    For \eqref{Dhata.d2lip}, we compute
     \begin{align*}
         D_j \hat{a}^{N,i} = D_{x^j} a^{N,i} + \sum_l D_{p^l} a^{N,i} D_{jl} u^{N,l}.
     \end{align*}
     Thus, 
     \begin{align*}
         |D_j \hat{a}^{N,i}|^2 &\leq C \Big( |D_{x^j} a^{N,i}|^2 + \big|\sum_{l \neq i,j} D_{p^l} a^{N,i} D_{jl}u^{N,l} \big|^2 + |D_{p^i} a^{N,i}|^2 |D_{ji} u^{N,i}|^2 + |D_{p^j} a^{N,i}|^2 |D_{jj} u^{N,j}|^2 \Big)
         \\
         &\leq C \Big( |D_{x^j} a^{N,i}|^2 + \big|\sum_{l \neq i,j} D_{p^l} a^{N,i} D_{jl}u^{N,l} \big|^2 +  |D_{ji} u^{N,i}|^2 + |\bA^N|_{\op}^2 |D_{p^j} a^{N,i}|^2  \Big),
     \end{align*}
     where we have used that $|D_{p^i} a^{N,i}|$ is uniformly bounded and $|D_{jj} u^{N,j}|\le |\bA^N|_{\op}.$
     It follows that 
     \begin{align*}
         \sum_{j \neq i} |D_j \hat{a}^{N,i}|^2 &\leq C \Big( \sum_{j \neq i} |D_{x^j} a^{N,i}|^2 + \sum_{j \neq i} \big| \sum_{l \neq i,j} D_{p^l} a^{N,i} D_{jl} u^{N,l} \big|^2 + \sum_{j \neq i} |D_{ji} u^{N,i}|^2 + |\bA^N|_{\op}^2 \sum_{j \neq i} |D_{p^j} a^{N,i}|^2 \Big)
         \\
         &\leq C(1+|\bA^N|_{\op}^2)/N  + C\sum_{j \neq i}  \big( \sum_{l \neq i} |D_{p^l} a^{N,i}|^2 \big) \big( \sum_{l \neq j} |D_{jl} u^{N,l}|^2  \big) + C \sum_{j \neq i} |D_{ji} u^{N,i}|^2
         \\
         &\leq C(1+|\bA^N|_{\op}^2)/N + C \sum_{j \neq i} |D_{ji} u^{N,i}|^2 + \frac{C}{N} \sum_{j \neq i} \sum_{l \neq j} |D_{jl} u^{N,l}|^2.
     \end{align*}
     This gives \eqref{Dhata.d2lip}.
\end{proof}

Next, we will obtain a bound on the off-diagonal derivatives. We first note that because $\bu^N$ has bounded second derivatives by assumption, it is a straightforward consequence of the (local) Calder\'on--Zygmund estimates for linear parabolic equations (see e.g. \cite[Chapter 4.9]{lady}) that in fact $u^{N,i} \in W_{t,\bx, \text{loc}}^{3,p}$ for each $p < \infty$. In particular, the derivatives 
$$u^{N,i,j} := D_j u^{N,i},$$ 
belong to $W_{t,\bx,\text{loc}}^{2,p}$ for each $p < \infty$, i.e. the weak derivatives $\partial_t u^{N,i,j}$, $(D_{x^k} u^{N,i,j})_{k = 1,\dots,N}$ and $D_{x^kx^l} u^{N,i,j})$ all belong to $L_{t,\bx}^p$ for each $p < \infty$. Moreover, by differentiating \eqref{nashsystem} with respect to $x^{j}$, we find 
\begin{align} \label{uNij.eqn}
    \partial_t u^{N,i,j} + \sum_{k = 1}^N \Delta_k u^{N,i,j} + \sigma_0 \sum_{k,l = 1}^N \tr\big( D_{kl} u^{N,i,j} \big)= F^{N,i,j}
\end{align}
as elements of $L_{t,\bx}^p$, where the operator $\partial_t + \sum_{k = 1}^N \Delta_k + \sigma_{0}\sum_{k,l = 1}^N \tr\big( D_{kl} \big)$ is applied component-wise to the $\R^d$-valued function $u^{N,i,j}$, and $F^{N,i,j} : [0,T] \times (\R^d)^N \to \R^d$ is the locally bounded function
\begin{align} \label{F.def}
       \ds F^{N,i,j} &= \sum_{k = 1}^N D_k u^{N,i,j} D_p H\big(x^k, u^{N,k,k}, m_{\bx,\hat{\ba}^N}^{N,-k} \big) 
       \nonumber \\
        \ds &\quad + {\bf 1}_{i = j} D_x H(x^i, u^{N,i,i}, m_{\bx, \hat{\ba}^N}^{N,-i}) + {\bf 1}_{i \neq j} D_{px} H\big(x^j, u^{N,j,j}, m_{\bx, \hat{\ba}^N}^{N,-j} \big) u^{N,i,j}
      \nonumber \\
       \ds &\quad +  \frac{{\bf{1}}_{i \neq j}}{N-1} D_{\mu}^x H\big(x^i, u^{N,i,i}, m_{\bx, \hat{\ba}^N}^{N,-i}, x^j,a^j \big) + \frac{1}{N-1} \sum_{k \neq i,j} \Big(D_{\mu}^x  D_p  H\big( x^k, u^{N,k,k}, m_{\bx, \hat{\ba}^N}^{N,-k}, x^j,a^j\big) \Big)^\top u^{N,i,k}
      \nonumber \\
       \ds &\quad + \sum_{k \neq i} \big(D_{jk} u^{N,k}\big)^\top D_{pp} H\big(x^k, u^{N,k,k}, m_{\bx,\hat{\ba}^N}^{N,-k} \big) u^{N,i,k}
      \nonumber \\
       \ds &\quad + \frac{1}{N-1} \sum_{k \neq i} \big(D_{j} \hat{a}^{N,k}\big)^\top D_{\mu}^a H\big(x^i,D_iu^{N,i}, m_{\bx,\hat{\ba}^N}^{N,-i}, x^k \hat{a}^{N,k}\big) 
      \nonumber \\
       \ds &\quad + \frac{1}{N-1} \sum_{k \neq i} \sum_{l \neq k} \big(D_j \hat{a}^{N,l}\big)^\top  \Big(D_{\mu}^a D_p   H \big(x^k, D_ku^{N,k}, m_{\bx, \hat{\ba}^N}^{N,-k}, x^l, \hat{a}^{N,l} \big) \Big)^\top  u^{N,i,k}.
\end{align}
Above, we are using standard conventions for products of matrices and vectors. We are viewing mixed second derivatives as $d \times d$ matrices by identifying the first subscript with the column number, so that e.g. 
\begin{align*}
    (D_{px} H)_{q,r} = \big(D_p D_x H \big)_{q,r} = D_{x_q p_r} H.
\end{align*}
We recall that since $\mu \in \cP_2(\R^d \times \R^d)$, $D_{\mu} H$ takes values in $\R^d \times \R^d$, and we are writing $D_{\mu} H = (D_{\mu}^x H, D_{\mu}^a H) \in \R^d \times \R^d$. Moreover, $D_j \hat{a}^{N,l}$ represents the Jacobian of $\hat{a}^{N,l}$ in the direction $j$, i.e. a $d \times d$ matrix of the form $(D_j \hat{a}^{N,l})_{q,r} = D_{j_r} \hat{a}^{N,l_q}$. So, for example, we have
\begin{align*}
    \Big( &\big(D_j \hat{a}^{N,l}\big)^\top \Big( D_{\mu}^a D_p  H  \big(x^k, D_ku^{N,k}, m_{\bx, \hat{\ba}^N}^{N,-k}, x^l, \hat{a}^{N,l} \big) \Big)^\top  D_k u^{N,i} \Big)_q 
    \\
    &= \sum_{r,s = 1}^d D_{j_q} \hat{a}^{N,l_r}   D_{\mu}^{a_r} D_{p_s} H \big(x^k, D_ku^{N,k}, m_{\bx, \hat{\ba}^N}^{N,-k}, x^l \hat{a}^{N,l} \big) D_{k_s}u^{N,i}.
\end{align*}
In addition, we have the terminal condition 
\begin{align} \label{unij.terminal}
       u^{N,i,j}(T, \bx) = {\bf1}_{i = j} D_x G(x^i,m_{\bx}^{N,-i}) + \frac{{\bf1}_{i \neq j}}{N-1} D_m G(x^i,m_{\bx}^{N,-i},x^j).
\end{align}

\begin{lemma} \label{lem.offdiagbound}
    Suppose that Assumptions \ref{assump.regularity} and \ref{assump.monotone} hold. Then there are constants $C>0, N_0\in\N$ with the following property. If $N \geq N_0$ and \eqref{startingpoint} holds for some $T_0 \in [0,T)$, then we have
    \begin{align*}
       \sup_{T_0 \leq t_0 \leq T, \bx_0 \in (\R^d)^N} \sum_{j \neq i} |D_j u^{N,i}|^2  +  \sup_{T_0 \leq t_0 \leq T, \bx_0 \in (\R^d)^N} \E\bigg[ \int_{t_0}^T \sum_{j \neq i} \sum_{k = 1}^N |D_{jk} u^{N,i}(t,\bX^{t_0,\bx_0}_t)|^2 dt \bigg]
       \\
       \leq C \exp( CK)/N.
    \end{align*}
\end{lemma}

\begin{proof}
    We fix $t_0,\bx_0$, and set 
    \begin{align*}
        \bX = \bX^{t_0,\bx_0}, \quad Y_t^{i,j} = u^{N,i,j}(t,\bX_t), \quad Z_t^{i,j,k} = \sqrt{2} D_k u^{N,i,j}(t,\bX_t), \quad {Z}_t^{i,j,0} = \sqrt{2\sigma_0} \sum_{k = 1}^N D_k u^{N,i,j}(t,\bX_t).
    \end{align*}
   By It\^o's formula, we find that for $i \neq j$,
    \begin{align*}
        d &\sum_{j \neq i} |Y_t^{i,j}|^2 = \sum_{j \neq i} \sum_{k = 0}^N |Z_t^{i,j,k}|^2 dt 
        \\
        &+ 2\sum_{j \neq i} \big(Y_t^{i,j}\big)^\top \bigg( D_{px} H\big(X_t^j, Y_t^{j,j}, m_{\bX_t, \hat{\ba}^N_t}^{N,-j} \big) Y_t^{i,j} + \frac{1}{N-1} D_{\mu}^x H\big( X_t^i, Y_t^{i,i}, m_{\bX_t, \hat{\ba}^N_t}^{N,-i}, X_t^j, \hat{a}^{N,j}_t  \big)
        \\
        &\quad + \frac{1}{N-1} \sum_{k \neq i,j} \Big( D_{\mu}^x  D_p H\big( X_t^k, Y_t^{k,k}, m_{\bX_t, \hat{\ba}^N_t}^{N,-k}, X_t^j, \hat{a}_t^{N,j}  \big) \Big)^\top Y_t^{i,k} 
        \\
        &\quad + \sum_{k \neq i} \big(A^{N,k,j}_t \big)^\top D_{pp} H\big(  X_t^k, Y_t^{k,k}, m_{\bX_t, \hat{\ba}^N_t}^{N,-k} \big) Y_t^{i,k}
        \\
        &\quad + \frac{1}{N-1} \sum_{k \neq i} \big(D_j \hat{a}^{N,k}_t \big)^\top D_{\mu}^a H\big( X_t^i, Y_t^{i,i}, m_{\bX_t, \hat{\ba}^N_t}^{N,-i}, X_t^k, \hat{a}^{N,k}_t  \big)  
        \\
        &\quad + \frac{1}{N-1} \sum_{k \neq i} \sum_{l \neq k} \big(D_j \hat{a}^{N,l}_t \big)^\top  \Big(D_{\mu}^a D_p H \big(X_t^k, Y_t^{k,k}, m_{\bX_t, \hat{\ba}^N_t}^{N,-k}, X_t^l, \hat{a}^{N,l}_t \big)\Big)^\top  Y_t^{i,k} \bigg) dt + dM_t,
    \end{align*}
    with $M$ being a martingale. Here we use notation like $\hat{\ba}_t^N = \hat{\ba}^N(t,\bX_t)$ to indicate functions that are evaluated at $(t,\bX_t)$.  Integrating in time and taking expectations, we find that for any $t_1 \in (t_0,T)$, we have 
    \begin{align} \label{t1through6}
        \E\bigg[ \sum_{j \neq i} |Y_{t_0}^{i,j}|^2 + {\int_{t_{0}}^{t_{1}}} \sum_{j \neq i} \sum_k |Z_t^{i,j,k}|^2 dt \bigg] \leq C \E\bigg[ \sum_{j \neq i} |Y_{t_1}^{i,j}|^2 \bigg] + {\tilde C}\E\bigg[ \int_{t_0}^{t_1} \sum_{m = 1}^6 |T_t^m| dt \bigg], 
    \end{align}
    where 
    \begin{align*}
        &T_t^1 = \sum_{j \neq i} (Y_t^{i,j})^\top D_{px} H(X_t^j, Y_t^{j,j}, m_{\bX_t, \hat{\ba}^N(t,\bX_t)}^{N,-j})Y_t^{i,j}, 
        \\
        &T_t^2 = \frac{1}{N-1} \sum_{j \neq i} (Y_t^{i,j})^\top D_{\mu}^x H\big( X_t^i, Y_t^{i,i}, m_{\bX_t, \hat{\ba}^N_t}^{N,-i}, X_t^j, \hat{a}^{N,j}_t  \big), 
        \\
        &T_t^3 = \frac{1}{N-1} \sum_{j \neq i} \sum_{k \neq i,j} (Y_t^{i,j})^\top  D_{\mu}^x D_p H\big( X_t^k, Y_t^{k,k}, m_{\bX_t, \hat{\ba}^N_t}^{N,-k}, X_t^j, \hat{a}^{N,j}_t  \big) Y_t^{i,k}, 
        \\
        &T_t^4 = \sum_{j \neq i} \sum_{k \neq i} (A^{N,k,j}(t,\bX_t) Y_t^{i,j})^\top  D_{pp} H\big(X_t^k,Y_t^{k,k}, m_{\bX_t,\hat{\ba}^N_t}^{N,-k}\big) Y_t^{i,k}
        \\
        &T_t^5 = \frac{1}{N-1} \sum_{j \neq i} \sum_{k \neq i} (Y_t^{i,j})^\top \big(D_j \hat{a}^{N,k}_t\big)^\top D_{\mu}^a H\big( X_t^i, Y_t^{i,i}, m_{\bX_t, \hat{\ba}^N_t}^{N,-i}, X_t^k, \hat{a}^{N,k}_t \big)   
        \\
        &T_t^6 = \frac{1}{N-1} \sum_{j \neq i} \sum_{k \neq i} \sum_{l \neq k} (Y_t^{i,j})^\top \big(D_j \hat{a}^{N,l}_t\big)^\top  \Big(D_{\mu}^a D_p  H \big(X_t^k, Y_t^{k,k}, m_{\bX_t, \hat{\ba}^N_t}^{N,-k}, X_t^l \hat{a}^{N,l}_t \big) \Big)^\top  Y_t^{i,k}.
    \end{align*}
    Using the boundedness of $D_{px} H$, we easily find that 
    \begin{align} \label{T1bound}
        |T_t^1| \leq C \sum_{j \neq i} |Y_t^{i,j}|^2.
    \end{align}
    By the boundedness of $D_{\mu}^xH$ and the Cauchy--Schwarz and Young inequalities, we find that 
    \begin{align} \label{T2bound}
        |T_t^2| \leq \frac{C}{\sqrt{N}} \Big(\sum_{j \neq i} |Y_t^{i,j}|^2 \Big)^{1/2} \leq \frac{C}{N} + C \sum_{j \neq i} |Y_t^{i,j}|^2.
    \end{align}
    Next, by boundedness of $D_{\mu}^x D_p H$, we find 
    \begin{align} \label{T3bound}
        |T_t^3| \leq \frac{C}{N} \sum_{j \neq i} \sum_{k \neq i} |Y_t^{i,j}| |Y_t^{i,k}| \leq C \sum_{j \neq i} |Y_t^{i,j}|^2.
    \end{align}
    For the fourth term, we use the boundedness of $D_{pp} H$ to get
    \begin{align} \label{T4bound}
        |T_t^4| \leq C |\bA^N_t|_{\op} \sum_{j \neq i} |Y_t^{i,j}|^2. 
    \end{align}
    For the fifth term, we use Lemma \ref{lem.hataprops} and the boundedness of $D_{\mu}^{a} H$ to get 
    \begin{align*} 
        |T_t^5| &\leq \frac{C}{N} \sum_{j \neq i} |Y_t^{i,j}| |D_j \hat{a}^{N,j}| + \frac{C}{N} \sum_{j \neq i} \sum_{k \neq j,i} |Y_t^{i,j}| |D_j \hat{a}^{N,k}| 
        \nonumber \\
        &\leq \frac{C}{\eps} \sum_{j \neq i} |Y_t^{i,j}|^2 + \frac{C}{N} \big(1 + |\bA^N_t|^2_{\op}\big) + \frac{C \eps}{N} \sum_{j=1}^N \sum_{k \neq j} |D_j \hat{a}^{N,k}|^2 
       \nonumber  \\
        &\leq \frac{C}{\eps} \sum_{j \neq i} |Y_t^{i,j}|^2 + \frac{C}{N} \big(1 + |\bA^N_t|^2_{\op}\big) + \frac{C \eps}{N} \Big( 1 + |\bA^N_t|_{\op}^2 + \sum_{j = 1}^N \sum_{k \neq j} |Z^{k,j,k}|^2  \Big).
    \end{align*}
    Choosing $\eps$ small enough, we get 
    \begin{align} \label{T5bound}
           |T_t^5| \leq C \sum_{j \neq i} |Y_t^{i,j}|^2 + \frac{C}{N} \big(1 + |\bA^N_t|_{\op}^2\big) +\frac{C}{N} + \frac{1}{2N {\tilde C}} \sum_{j = 1}^N \sum_{k \neq j} |Z_t^{k,j,k}|^2,
    \end{align}
 {where $\tilde C$ is the constant appearing in \eqref{t1through6}.} Finally for the last term, we use Lemma \ref{lem.hataprops} and the boundedness of $D_{\mu}^a D_p H$ to get 
     \begin{align} \label{T6bound}
            |T_t^6| &\leq \frac{C}{N} \sum_{k \neq i} |Y_t^{i,k}| \left| \sum_{j \neq i} \sum_{l \neq k} (Y_t^{i,j})^\top \big(D_j \hat{a}^{N,l}_t \big)^\top \Big(D_{\mu}^a  D_p H\big( X_t^k, Y_t^{k,k},m_{\bX_t, \hat{\ba}^N_t}^{N,-k},X_t^l, \hat{a}^{N,l}_t \big) \Big)^\top \right| \nonumber
            \\
            &\leq \frac{C}{N} \sum_{k \neq i}|Y_t^{i,k}|  \Big( \sum_{j \neq i}|Y_t^{i,j}|^2 \Big)^{1/2} |D\hat{\ba}^{N}|_{\op} \Big( \sum_{l \neq k} |D_{\mu}^a D_p H \big( X_t^k, Y_t^{k,k},m_{\bX_t, \hat{\ba}^N_t}^{N,-k},\hat{a}^{N,l}_t \big)|^2 \Big)^{1/2}
            \nonumber\\ 
            &\leq \frac{C}{\sqrt{N}}\big(1 + |\bA^N_t|_{\op}\big)\Big( \sum_{j \neq i}|Y_t^{i,j}|^2 \Big)^{1/2} \sum_{k \neq i} |Y_t^{i,k}|\le C\big(1 + |\bA^N_t|_{\op}\big) \sum_{j \neq i} |Y_t^{i,j}|^2
    \end{align}
    Putting together \eqref{T1bound}, \eqref{T2bound}, \eqref{T3bound}, \eqref{T4bound}, \eqref{T5bound}, and \eqref{T6bound}, we arrive at 
    \begin{align*}
        {\tilde C}\sum_{m = 1}^N |T_t^m| &\leq \frac{C(1 + |\bA^N_t|^2_{\op})}{N} + C \big(1 + |\bA^N_t|_{\op} \big) \sum_{j \neq i} |Y_t^{i,j}|^2 + \frac{1}{2N} \sum_{j = 1}^N \sum_{k \neq j} |Z^{k,j,k}_t|^2
        \\
        &\leq \frac{C(1 + |\bA^N_t|^2_{\op})}{N} + C \big(1 + |\bA^N_t|_{\op} \big) \sum_{j \neq i} |Y_t^{i,j}|^2 + \frac{1}{2N} \sum_{j = 1}^N \sum_{k \neq j} \sum_{l = 1}^N |Z^{k,j,l}_t|^2.
    \end{align*}
{Here, let us emphasize that it is crucial to choose $\eps>0$ in the estimate for $T^{5}_{t}$ small enough that we have precisely the factor $1/(2N)$ in the last term of the previous chain of inequalities.} Coming back to \eqref{t1through6}, {using the assumption \eqref{startingpoint},} we deduce that for any $T_0 \leq t_0 < t_1 \leq T$, we have
    \begin{align*}
     \E\bigg[ \sum_{j \neq i} &|Y_{t_0}^{i,j}|^2 + \int_{t_0}^{t_1}  \sum_{j \neq i} \sum_k |Z_t^{i,j,k}|^2 dt \bigg] \leq \E\big[ \sum_{j \neq i} |Y_{t_1}^{i,j}|^2 \big]  \\
     &\qquad +\E\bigg[ \int_t^T \frac{C(1 + |\bA^N_s|^2_{\op})}{N}
     + C |\bA^N_s| \sum_{j \neq i} |Y_t^{i,j}|^2 + \frac{1}{2N} \sum_{j = 1}^N \sum_{k \neq j} |Z^{k,j,k}_t|^2 \Big) dt \bigg]
        \\
        &\leq  \E\big[ \sum_{j \neq i} |Y_{t_1}^{i,j}|^2 \big] + \E\bigg[  {\int_{t_0}^{t_1}} \Big(\frac{C(1 + |\bA^N_s|^2_{\op})}{N} + C \big(1 + |\bA^N_s| \big)\sum_{j \neq i} |Y_t^{i,j}|^2 
       + \frac{1}{2N} \sum_{j = 1}^N \sum_{k \neq j} \sum_{l = 1}^N |Z^{k,j,l}_t|^2\Big) dt \bigg]
        \\
        &\leq  \E\big[ \sum_{j \neq i} |Y_{t_1}^{i,j}|^2 \big] + C(1+{K})/N + C \Big\|\sum_{j \neq i} |u^{N,i,j}|^2 \Big\|_{L^{\infty}([t_0,t_1] \times (\R^d)^N} \E\bigg[ {\int_{t_0}^{t_1}} \big(1 + |\bA^N_t| \big) dt \bigg] 
        \\
        &\qquad + \frac{1}{2N} \E\bigg[ \int_{t_0}^{t_1} \sum_{j = 1}^N \sum_{k \neq j} \sum_{l = 1}^N |Z^{k,j,l}_t|^2dt \Big) \bigg]
        \\
        &\leq  \E\big[ \sum_{j \neq i} |Y_{t_1}^{i,j}|^2 \big] + C(1+{K})/N + C \Big\|\sum_{j \neq i} |u^{N,i,j}|^2 \Big\|_{L^{\infty}([t_0,t_1] \times (\R^d)^N}{\left(t_{1}-t_{0} + \sqrt{K} \sqrt{t_1-t_0}\right)} 
        \\
        &\qquad + \frac{1}{2N} \E\bigg[ \int_{t_0}^{t_1} \sum_{j = 1}^N \sum_{k \neq j} \sum_{l = 1}^N |Z^{k,j,l}_t|^2dt \Big) \bigg].
    \end{align*}
    Now take a maximum over $i = 1,\dots,N$ and then absorb the last term on the right hand-side to get 
    \begin{align} \label{zijkest}
       \max_{i = 1,\dots,N} \E\bigg[ &\sum_{j \neq i} |Y_{t_0}^{i,j}|^2 + \int_{t_0}^{t_1} \sum_{j \neq i} \sum_k |Z_t^{i,j,k}|^2 dt \bigg] \leq \Big\|\sum_{j \neq i} |u^{N,i,j}({t_1},\cdot)|^2 \Big\|_{L^{\infty}((\R^d)^N)} + C(1+K)/N
       \nonumber \\
       &+ C \max_{i = 1,\dots,N} \Big\|\sum_{j \neq i} |u^{N,i,j}|^2 \Big\|_{L^{\infty}([t_0,T] \times (\R^d)^N} \left(t_{1}-t_{0} + \sqrt{K} \sqrt{t_1-t_0}\right) . 
    \end{align}
    Recall the definition of $Y^{i,j}$ and take a supremum over $t_0 \in [t_1 - \eps, t_1]$ {(for $\eps>0$ small, to be chosen later)} and $\bx_0 \in (\R^d)^N$ to find that for any $t_1 \in [T_0, T]$, we have
    \begin{align*}
        &\max_{i = 1,\dots,N}  \Big\|\sum_{j \neq i} |u^{N,i,j}|^2 \Big\|_{L^{\infty}([t_1-\eps,t_1] \times (\R^d)^N)} \leq \max_{i = 1,...,N} \Big\|\sum_{j \neq i} |u^{N,i,j}(t_1,\cdot)|^2 \Big\|_{L^{\infty}((\R^d)^N)} + C(1+K)/N
        \\
        &\qquad + C \max_{i = 1,\dots,N} \Big\|\sum_{j \neq i} |u^{N,i,j}|^2 \Big\|_{L^{\infty}([t_1 - \eps,t_1] \times (\R^d)^N)} {\left(\eps+\sqrt{K} \sqrt{\eps}\right)}.
    \end{align*}
{As for $\eps\le 1$ we have that $\eps\le\sqrt{\eps}$, and so in this case $\eps+\sqrt{K} \sqrt{\eps} \le (1+\sqrt{K})\sqrt{\eps}$, we have that for any $\eps \leq \min\left\{1,\frac{1}{C^2 4 (1+\sqrt{K})^{2}}\right\}$} and any $t_1$ with $T_0 + \eps \leq t_1 \leq T$, we have
    \begin{align*}
          \max_{i = 1,\dots,N}  \Big\|\sum_{j \neq i} |u^{N,i,j}|^2 \Big\|_{L^{\infty}([t_1-\eps,t_1] \times (\R^d)^N)} \leq 2\max_{i = 1,\dots,N} \Big\|\sum_{j \neq i} |u^{N,i,j}(t_1,\cdot)|^2 \Big\|_{L^{\infty}((\R^d)^N)} + C(1+K)/N.
    \end{align*}
 {By \eqref{unij.terminal} and by our standing assumptions on $D_{m}G$, we have that 
 $$
 \Big\|\sum_{j \neq i} |u^{N,i,j}(T,\cdot)|^2\Big\|_{L^{\infty}((\R^d)^N)} \le C/N,
 $$
 so, iterating the previous} inequality $\lceil (T-T_0) C^2 4{(1+\sqrt{K})^{2}} \rceil$ times, we get the desired bound on the quantity $\big\| \sum_{j \neq i} |u^{N,i,j}|^2\big\|_{\linf([T_0,T] \times (\R^d)^N)}$, and the bound on $D_{jk} u^{N,i}$ can then be obtained from taking $t_1 = T$ in \eqref{zijkest}, after recalling the definition of $Z^{i,j,k}$. 
\end{proof}

\begin{proposition}
    \label{prop.uiivi}
    Suppose that Assumptions \ref{assump.regularity} and \ref{assump.monotone} hold. Then there are constants $C>0, N_0\in\N$ with the following property. If $N \geq N_0$ and \eqref{startingpoint} holds for some $T_0 \in [0,T)$, then 
    \begin{align*}
        \norm{ \sum_{i = 1}^N |v^{N,i} - D_i u^{N,i}|^2 }_{\linf([T_0,T] \times (\R^d)^N)} & + \sup_{T_0 \leq t_0 \leq T, \bx_0 \in (\R^d)^N} \E\bigg[ \int_{t_0}^T \sum_{i,j = 1,\dots,N} |D_{ji} u^{N,i} - D_j v^{N,i}|^2(t,\bX^{t_0,\bx_0}_t) dt \bigg] 
        \\
       & \leq \frac{C \exp(CK)}{N}.
    \end{align*}
\end{proposition}

\begin{proof}
    The starting point is to fix $t_0 \in [T_0,T)$, $\bx_0 \in (\R^d)^N$, and set 
    \begin{align*}
        &\bX_t = \bX_t^{t_0,\bx_0}, \quad Y_t^{i} = u^{N,i,i}(t,\bX_t), \quad Z_t^{i,j} = \sqrt{2} D_j u^{N,i,i}(t,\bX_t), \quad Z_t^{i,0} = \sqrt{2\sigma_0} \sum_{k = 1}^N D_k u^{N,i,i}(t,\bX_t),
        \\
        &\ov{Y}_t^i = v^{N,i}(t,\bX_t), 
       \quad \ov{Z}_t^{N,i,j} = \sqrt{2} D_j v^{N,i}(t,\bX_t), \quad \ov{Z}_t^{N,i,0} = \sqrt{2\sigma_0} \sum_{k = 1}^N D_k v^{N,i}(t,\bX_t). 
    \end{align*}
    Then we have 
    \begin{align*}
        dY_t^{i} &= \Big( D_x H(X_t^i, Y_t^{i}, m_{\bX_t, \ba^N(\bX_t,\bY_t)}^{N,-i})
       + \sum_{m = 1}^4 T_t^m \Big) dt  + \sum_{j = 0}^N Z_t^{i,j} dW_t^j,
    \end{align*}
    where 
    \begin{align} \label{Tdef}
        \nonumber &T_t^{1,i} =  \frac{1}{N-1} \sum_{k \neq i}  \Big(D_{\mu}^x D_p  H(X_t^k,Y_t^{k}, m_{\bX_t, \hat{\ba}^N(t,\bX_t}^{N,-k}), X_t^i, \hat{a}^{N,i}(t,\bX_t))\Big)^\top D_k u^{N,i}(t,\bX_t), 
        \\
        &T_t^{2,i} = \sum_{k \neq i} \big(D_{ik}u^{N,k}(t,\bX_t) \big)^\top D_{pp}H(X_t^k, Y_t^k, m_{\bX_t, \hat{\ba}^N(t,\bX_t)}^{N,-k})D_k u^{N,i}(t,\bX_t), 
       \nonumber \\
        &T_t^{3,i} = \frac{1}{N-1} \sum_{k \neq i}  \big(D_i \hat{a}^{N,k}(t,\bX_t)\big)^\top D_{\mu}^aH(X_t^i, Y_t^i,m_{\bX_t, \hat{\ba}^N(\bX_t)}^{N,-i},X_t^k, \hat{a}^{N,k}(t,\bX_t)), 
       \nonumber \\
        &T_t^{4,i} = \frac{1}{N-1} \sum_{k \neq i} \sum_{l \neq k} \big(D_i \hat{a}^{N,l}(t,\bX_t)\big)^\top   \Big(D_{\mu}^a D_p  H\big( X_t^k, Y_t^k, m_{\bX_t, \hat{\ba}^N(t,\bX_t)}^{N,-k}, X_t^l, \hat{a}^{N,l}(t,\bX_t) \big)\Big)^T D_k u^{N,i}(t,\bX_t) .
    \end{align}
    Similarly, 
    \begin{align*}
        d \ov{Y}_t^i &= \Big( D_x H(X_t^i, \ov{Y}_t^{i}, m_{\bX_t, \ba^N(\bX_t,\ov{\bY}_t)}^{N,-i}) 
        \\
        &\quad - \sum_{j =1}^N D_j v^{N,i}(t,\bX_t) \big( D_pH(X_t^j, Y_t^j, m_{\bX_t, \ba^N(\bX_t, \bY_t)}^{N,-j}) - D_pH(X_t^j, \ov{Y}_t^j, m_{\bX_t, \ba^N(\bX_t, \ov{\bY}_t)}^{N,-j}) \big) \Big) dt + \sum_{j = 0}^N \ov{Z}_t^{i,j} dW_t^j.
    \end{align*}
    We now set 
    \begin{align*}
        \bm{\Delta Y} = (\Delta Y_t^i)_{i = 1,...,N}, \quad \Delta Y_t^i := Y_t^i - \ov{Y}_t^i, \quad \bm{\Delta Z} = (\Delta Z_t^{i,j})_{i,j = 1,...,N}, \quad {\Delta}Z_t^{i,j} := Z_t^{i,j} - \ov{Z}_t^{i,j}, 
    \end{align*}
    and we compute 
    \begin{align*}
        d  \big|\bm{\Delta Y_t}\big|^2 &= 2\bigg( \sum_{i = 1}^N (\Delta Y_t^i)^\top \Big( D_x H(X_t^i, Y_t^{i}, m_{\bX_t, \ba^N(\bX_t,\bY_t)}^{N,-i}) - D_x H(X_t^i, \ov{Y}_t^{i}, m_{\bX_t, \ba^N(\bX_t,\ov{\bY}_t)}^{N,-i}) \Big) 
        \\
        &\quad + \sum_{i,j = 1}^N (\Delta Y_t^i)^\top D_j v^{N,i}(t,\bX_t) \Big( D_pH(X_t^j, Y_t^j, m_{\bX_t, \ba^N(\bX_t, \bY_t)}^{N,-j}) - D_pH(X_t^j, \ov{Y}_t^j, m_{\bX_t, \ba^N(\bX_t, \ov{\bY}_t)}^{N,-j}) \Big)  
        \\
        &\quad + \sum_{m = 1}^4 (\Delta Y_t^i)^\top T_t^{m,i} \bigg) + \big|\bm{\Delta Z}_t \big|^2 dt + dM_t
    \end{align*}
    with $M$ being a martingale. In particular, for any $t_0 \leq t \leq T$, we find 
    \begin{align*}
        \E\bigg[ &\big|\bm{\Delta Y}\big|^2 + \int_t^T \big|\bm{\Delta Z}_s\big|^2 ds \bigg] 
        \\
        &\leq -\E\bigg[ \int_{t}^T \sum_{i = 1}^N (\Delta Y_s^i)^\top \Big( D_x H(X_s^i, Y_s^{i}, m_{\bX_s, \ba^N(\bX_s,\bY_s)}^{N,-i}) - D_x H(X_s^i, \ov{Y}_s^{i}, m_{\bX_s, \ba^N(\bX_s,\ov{\bY}_s)}^{N,-i}) \Big) 
        \\
        &\qquad + \sum_{i,j = 1}^N (\Delta Y_s^i)^\top D_j v^{N,i}(s,\bX_s) \Big( D_pH(X_s^j, Y_s^j, m_{\bX_s, \ba^N(\bX_s, \bY_s)}^{N,-j}) - D_pH(X_s^j, \ov{Y}_s^j, m_{\bX_s, \ba^N(\bX_s, \ov{\bY}_s)}^{N,-j}) \Big)  
        \\
        &\qquad + \sum_{m = 1}^4 \sum_{i = 1}^N (\Delta Y_s^i)^{\top} T_s^{m,i} \Big)ds   \bigg].
    \end{align*}
    Next, notice that 
    \begin{align*}
       \Big| \sum_{i = 1}^N &(\Delta Y_s^i)^\top \Big( D_x H(X_s^i, Y_s^{i}, m_{\bX_s, \ba^N(\bX_s,\bY_s)}^{N,-i}) - D_x H(X_s^i, \ov{Y}_s^{i}, m_{\bX_s, \ba^N(\bX_s,\ov{\bY}_s)}^{N,-i}) \Big) \Big| 
       \\ &\leq \sum_{i = 1}^N |\Delta Y_s^i|^2 + \sum_{i = 1}^N \Big| D_x H(X_s^i, Y_s^{i}, m_{\bX_s, \ba^N(\bX_s,\bY_s)}^{N,-i}) - D_x H(X_s^i, \ov{Y}_s^{i}, m_{\bX_s, \ba^N(\bX_s,\ov{\bY}_s)}^{N,-i}) \Big|^2
       \\
       &\leq C \sum_{i = 1}^N |\Delta Y_s^i|^2 + C \sum_{i = 1}^N \bd_2^2\big(m_{\bX_s, \ba^N(\bX_s,\bY_s)}^{N,-i}, m_{\bX_s, \ba^N(\bX_s,\ov{\bY}_s)}^{N,-i} \big)
       \\
       &\leq C \sum_{i = 1}^N |\Delta Y_s^i|^2 + C \sum_{i = 1}^N |a^{N,i}(\bX_s, \bY_s) - a^{N,i}(\bX_s, \ov{\bY}_s)|^2 \leq C \sum_{i = 1}^N |\Delta Y_s^i|^2,
    \end{align*}
{where we have used the Lipschitz continuity of $D_{x}H$ and Lemma \ref{lem.ani}.}    
    Similarly, using the Lipschitz bound on $\bv^N$ from Proposition \ref{prop.vLip}, we get 
    \begin{align*}
        \Big| \sum_{i,j = 1}^N & (\Delta Y_s^i)^\top D_j v^{N,i}(s,\bX_s) \Big( D_pH(X_s^j, Y_s^j, m_{\bX_s, \ba^N(\bX_s, \bY_s)}^{N,-j}) - D_pH(X_s^j, \ov{Y}_s^j, m_{\bX_s, \ba^N(\bX_s, \ov{\bY}_s)}^{N,-j}) \Big|
        \\
        &\leq C \big( \sum_{i = 1}^N |\Delta Y_s^i|^2 \big)^{1/2} \Big( \sum_{i = 1}^N \big|D_pH(X_s^i, Y_s^i, m_{\bX_s, \ba^N(\bX_s, \bY_s)}^{N,-i}) - D_pH(X_s^i, \ov{Y}_s^i, m_{\bX_s, \ba^N(\bX_s, \ov{\bY}_s)}^{N,-i}) \big|^2 \Big)^{1/2}
        \\
        &\leq C\sum_{i = 1}^N |\Delta Y_s^i|^2 + C \sum_{i = 1}^N \bd_2^2\big(m_{\bX_s, \ba^N(\bX_s,\bY_s)}^{N,-i}, m_{\bX_s, \ba^N(\bX_s,\ov{\bY}_s)}^{N,-i} \big) \leq C \sum_{i = 1}^N |\Delta Y_s^i|^2. 
    \end{align*}
    Together with an application of Young's inequality to handle the term $\sum_{ m =1}^4 \sum_{i = 1}^N (\Delta Y_s^i)^{\top} {T_s^{m,i}}$, this allows us to deduce 
    \begin{align*}
        \E\bigg[ & |\bm{\Delta Y_t}|^2 + \int_t^T \big|\bm{\Delta Z}_s\big|^2 ds  \bigg] 
        \\
        &\leq C\E\bigg[ \int_{t}^T \Big(|\bm{\Delta Y_s}|^2+ \sum_{m = 1}^4 \sum_{i = 1}^N |T_s^{m,i}|^2 \Big)ds   \bigg].
    \end{align*}
    Applying Gr\"onwall's inequality to the function $t \mapsto \E\big[ |\bm{\Delta Y_t}|^2\big]$, we deduce that
    \begin{align*}
        \E\bigg[ |\bm{\Delta Y_{t_0}}|^2 + \int_{t_0}^T \big|\bm{\Delta Z}_s\big|^2 ds  \bigg] 
       \leq C\E\bigg[ \int_{t_0}^T  \sum_{i = 1}^N |T_t^{m,i}|^2 dt  \bigg].
    \end{align*}
    It remains to estimate the error terms $T^{m,i}$ in $L^2$. We start by using Lemma \ref{lem.offdiagbound} and the boundedness of $ D_p  D_{\mu}^x H$ to get
    \begin{align*}
        \sum_{i = 1}^N |T_t^{1,i}|^2 &\leq \frac{C}{N^2} \sum_{i = 1}^N \Big|\sum_{k \neq i} D_p  D_{\mu}^x H\big( X_t^k,Y_t^k, m_{\bX_t, \hat{\ba}^N(\bX_t)}^{N,-k}, X_t^i \big) D_k u^{N,i}(t,\bX_t)\Big|^2
        \\
        &\leq \frac{C}{N} \sum_{i = 1}^N \sum_{k \neq i} |D_ku^{N,i}(t,\bX_t)|^2 \leq C \exp(CK)/N.
    \end{align*}
    Next, we have 
    \begin{align*}
      \sum_{i = 1}^N &|T_t^{2,i}|^2 = \sum_{i = 1}^N \Big|  \sum_{k \neq i} (D_{ik}u^{N,k}(t,\bX_t))^\top D_{pp}H(X_t^k, Y_t^k, m_{\bX_t, \hat{\ba}^N(t,\bX_t)}^{N,-k})D_k u^{N,i}(t,\bX_t) \Big|^2
       \\
       &\leq C\sum_{i = 1}^N \Big( \sum_{k \neq i} |D_{ik}u^{N,k}(t,\bX_t)|^2 \Big) \Big( \sum_{k \neq i} |D_k u^{N,i}(t,\bX_t)|^2 \Big) \leq \frac{C \exp(CK)}{N} \sum_{i = 1}^N \sum_{k \neq i} |D_{ik}u^{N,k}(t,\bX_t)|^2, 
    \end{align*}
    so that in particular
    \begin{align*}
        \E\bigg[ \int_{t_0}^T \sum_{i = 1}^N |T_t^{2,i}|^2 dt \bigg] \leq \frac{C \exp(CK)}{N} \E\bigg[ \sum_{i = 1}^N \sum_{j \neq i} \int_{t_0}^T |D_{ij}u^{N,i}(t,\bX_t)|^2 \bigg] \leq \frac{C \exp(CK)}{N}.
    \end{align*}
    Next, using Lemma \ref{lem.hataprops} and the boundedness of $D_{\mu}^a H$, we have

    \begin{align*}
        \sum_{ i = 1}^N |T_t^{3,i}|^2 &\leq  \frac{C}{N^2} \sum_{i = 1}^N \Big| \sum_{k \neq i} \big(D_i \hat{a}^{N,k}(t,\bX_t)\big)^\top D_{\mu}^aH(X_t^i, Y_t^i,m_{\bX_t, \hat{\ba}^N(t,\bX_t)}^{N,-i}, \ba^k(t,\bX_t))  \Big|^2
        \\
        &\leq \frac{C}{N} \sum_{i = 1}^N \sum_{k \neq i} |D_i \hat{a}^k(t,\bX_t)|^2\\
        &\leq \frac{C}{N} \Big( 1 + |\bA^N|^2_{\op} + \sum_{i = 1}^N \sum_{j \neq i} |D_{ij}u^{N,i}|^2(t,\bX_t) \Big), 
    \end{align*}
    so that 
    \begin{align*}
        \E\bigg[ \int_{t_0}^T \sum_{i = 1}^N |T_t^{3,i}|^2 dt \bigg] &\leq \frac{C}{N} \big(1 + {K}\big) + \frac{C}{N} \E\bigg[ \int_{t_0}^T  \sum_{i = 1}^N \sum_{j \neq i} |D_{ij}u^{N,i}|^2(t,\bX_t) dt \bigg]
        \\
        &\leq \frac{C \exp(CK)}{N}.
    \end{align*}
    Finally, we again use Lemma \ref{lem.hataprops} {and Lemma \ref{lem.offdiagbound}}  together with the boundedness of $ D_{\mu}^a D_p  H$ to get 
    \begin{align*}
        \sum_{i = 1}^N &|T_t^{4,i}|^2 \leq \frac{C}{N^2} \sum_{i = 1}^N \Big| \sum_{k \neq i} \sum_{l \neq k} \big(D_i \hat{a}^{N,l}(t,\bX_t)\big)^\top  \Big(D_{\mu}^a D_p H\big( X_t^k, Y_t^k, m_{\bX_t, \hat{\ba}^N(t,\bX_t)}^{N,-k}, \hat{a}^{N,l}(t,\bX_t) \big)\Big)^\top  D_k u^{N,i}(t,\bX_t)\Big) \Big|^2
        \\
        &\leq \frac{C}{N^2} \sum_{i = 1}^N \Big| \sum_{l= 1}^N |D_i \hat{a}^{N,l}(t,\bX_t)| \sqrt{N} \big(\sum_{k \neq i} |D_k u^{N,i}(t,\bX_t)|^2\big)^{1/2} \Big|^2
        \\
        &\leq \frac{C \exp(CK)}{N^2} \sum_{i = 1}^N \Big| \sum_{l = 1}^N |D_i \hat{a}^{N,l}(t,\bX_t)| \Big|^2
        \\
        &\leq \frac{C \exp(CK)}{N^2} \sum_{i = 1}^N |D_i \hat{a}^{N,i}(t,\bX_t)|^2 + \frac{C \exp(C{K})}{N} \sum_{i = 1}^N \sum_{j \neq i} |D_j \hat{a}^{N,i}(t,\bX_t)|^2
        \\
        &\leq \frac{C \exp(CK)}{N} \Big( 1 + \big|\bA^N(t,\bX_t)\big|^2 + \sum_{i = 1}^N \sum_{j \neq i} |D_{ij} u^{N,i}(t,\bX_t)|^2\Big), 
    \end{align*}
    so that
    \begin{align*}
         \E\bigg[ \int_{t_0}^T \sum_{i = 1}^N |T_t^{4,i}|^2 dt \bigg] &\leq \frac{C \exp(CK)}{N} \bigg(1 + \E\bigg[{\int_{t_0}^T}& \Big( \big|\bA^N(t,\bX_t) \big|^2 + \sum_{i = 1}^N \sum_{j \neq i} |D_{ij} u^{N,i}(t,\bX_t)|^2\Big) dt \bigg] \bigg)
         \\
         &\leq \frac{C \exp(CK)}{N}.
    \end{align*}
    We thus deduce that 
     \begin{align*}
        \E\bigg[ \sum_{i = 1}^N |\Delta Y_{t_0}^i|^2 + \int_{t_0}^T \big|\bm{\Delta Z}_t\big|^2 dt  \bigg] 
       \leq \frac{C \exp(CK)}{N}.
    \end{align*}
    Recalling the definition of $\Delta Y^i$ and $\Delta Z^{i,j}$, and taking a supremum over $t_0 \in [T_0,T]$, $\bx_0 \in (\R^d)^N$ completes the proof.
\end{proof}

\begin{lemma}
    \label{lem.improvementofM}
    Suppose that Assumptions \ref{assump.regularity} and \ref{assump.monotone} hold. Then there are constants $C>0, N_0\in \N$ with the following property. If $N \geq N_0$ and \eqref{startingpoint} holds for some $T_0 \in [0,T)$, then we have 
    \begin{align} 
    \sup_{T_0 \leq t_0 \leq T, \bx_0 \in (\R^d)^N} \E\bigg[ \int_{t_0}^T |\bA^N(t, \bX_t^{t_0,\bx_0})|_{\op}^2 dt\bigg] \leq C + \frac{C \exp(CK)}{N}.
\end{align}
\end{lemma}

\begin{proof}
    Using Proposition \ref{prop.uiivi} and Proposition \ref{prop.vLip}, we get 
    \begin{align*}
        \E\bigg[ \int_{t_0}^T |\bA^N(t,\bX_t)|_{\op}^2 dt \bigg] &\leq 2  \E\bigg[ \int_{t_0}^T |\bA^N(t,\bX_t) - D\bv^N(t,\bX_t)|_{\op}^2 dt \bigg] +  2 \E\bigg[ \int_{t_0}^T |D \bv^N(t,\bX_t)|_{\op}^2 dt \bigg]
        \\
        &\leq \frac{C \exp(CK)}{N} + C.
    \end{align*}
\end{proof}

\begin{proposition}
    \label{prop.closetheloop}
    Suppose that Assumptions \ref{assump.regularity} and \ref{assump.monotone} hold. Then there is a constant $C_0>0$ such that for all $N\in\N$ large enough, 
    \begin{align*}
         \sup_{{0} \leq t_0 \leq T, \bx_0 \in (\R^d)^N} \E\bigg[ \int_{t_0}^T |\bA^N(t, \bX_t^{t_0,\bx_0})|_{\op}^2 dt\bigg] \leq C_0.
    \end{align*}
\end{proposition}

\begin{proof}
    Let $C$ be as in the statement of Lemma \ref{lem.improvementofM}, and choose $N_0$ large enough that 
    \begin{align*}
        C + \frac{C \exp(3C^2)}{N_0} \leq 2C.
    \end{align*}
    Define 
    \begin{align*}
        T^* := \inf \bigg\{ T_0 \in [0,T] :  \sup_{T_0 \leq t_0 \leq T, \bx_0 \in (\R^d)^N} \E\bigg[ \int_{t_0}^T |\bA^N(t, \bX_t^{t_0,\bx_0})|_{\op}^2 dt\bigg]  \leq 2C \bigg\}.
    \end{align*}
    Suppose towards a contradiction that $T^* > 0$. Then because $\bA^N \in L^{\infty}$ (because by definition admissible solutions to the Nash system satisfy $D_{jk} u^{N,i} \in L^{\infty}$), for any $\eps > 0$ and $T^* - \eps \leq t_0 \leq T$, we have 
    
    \begin{align*}
        \E\bigg[ \int_{t_0}^T |\bA^N(t, \bX_t^{t_0,\bx_0})|_{\op}^2 dt\bigg] &=  \E\bigg[ \int_{t_0}^{T^*} |\bA^N(t, \bX_t^{t_0,\bx_0})|_{\op}^2 dt\bigg] +  \E\bigg[ \int_{T^*}^{T} |\bA^N(t, \bX_t^{t_0,\bx_0})|_{\op}^2 dt\bigg]
        \\
        &\leq C' \eps + \E\bigg[ \int_{T^*}^{T} |\bA^N(t, \bX_t^{t_0,\bX_{T^*}^{t_0,\bx_0}})|_{\op}^2 dt\bigg]
        \\
        &\leq C' \eps + 2C.
    \end{align*}
    In particular, we can find $\eps > 0$ such that
    \begin{align*}
        \sup_{T^*-\eps \leq t_0 \leq T, \bx_0 \in (\R^d)^N} \E\bigg[ \int_{t_0}^T |\bA^N(t, \bX_t^{t_0,\bx_0})|_{\op}^2 dt\bigg] \leq 3C. 
    \end{align*}
    But then by Lemma \ref{lem.improvementofM}, we get 
    \begin{align*}
         \sup_{T^*-\eps \leq t_0 \leq T, \bx_0 \in (\R^d)^N} \E\bigg[ \int_{t_0}^T |\bA^N(t, \bX_t^{t_0,\bx_0})|_{\op}^2 dt\bigg] \leq  C + \frac{C \exp(3C^2)}{N_0} \leq 2C, 
    \end{align*}
    which contradicts the definition of $T^*$. Thus the result holds with $C_0 = 3C$.
\end{proof}

We now complete the proof of Theorem \ref{thm.main}, by establishing the estimates \eqref{mainest3} and \eqref{mainest4}.

\begin{proof}[Proof of closed-loop convergence]
    Set $\hat{\bX} = \bX^{\cl,N}$, and $\hat{\bY} = (\hat{Y}^1,..,\hat{Y}^N)$, $\hat{\bZ} = (\hat{Z}^{i,j})_{i,j = 1,\dots,N}$ via
    \begin{align*}
        \hat{Y}_t^{i} = D_i u^{N,i}(t,\bX_t), \quad \hat{Z}_t^{i,j} = \sqrt{2} D_{ji} u^{N,i}(t,\bX_t), \quad \hat{Z}_t^{i,0} = \sqrt{2 \sigma_0} \sum_{k = 1}^N D_{ki} u^{N,i}(t,\bX_t).
    \end{align*}
    Following the computations in the proof of Proposition \ref{prop.uiivi}, we have that $\big(\hat{\bX}^N, \hat{\bY}^N, \hat{\bZ}^N \big)$ satisfies \eqref{pontryaginNplayer.gen.error} with errors 
    \begin{align*}
     \cE_t^{0,i} =   \cE_t^{1,i} = \cE_t^{3,i} = 0, \quad \cE_t^{2,i} = \sum_{m = 1}^4 T_t^{m,i},
    \end{align*}
    with $T^{m,i}$ defined as in \eqref{Tdef}. 
    By following the computation in the proof of Proposition \ref{prop.uiivi} and using Proposition \ref{prop.closetheloop} to bound $K$, we get 
    \begin{align*}
      \E\bigg[ \int_0^T \big|\bm{\cE}_t^{2}\big|^2 dt \bigg] \leq C \E\bigg[ \int_0^T \sum_{i = 1}^N \sum_{m = 1}^4 |T_t^{m,i}|^2 dt \bigg] \leq C/N.
    \end{align*}
    Thus \eqref{mainest3} follows from Proposition \ref{prop.uniformstability}, and then \eqref{mainest4} follows from combining \eqref{mainest3} and \eqref{mainest2}.
\end{proof}

\appendix

\section{Well-posedness of the mean field Pontryagin system}
Throughout this Appendix, Assumptions \ref{assump.monotone} and \ref{assump.regularity} are in force.
Let us denote by $\mathcal{S}^2 = \cS^2(\bbF; \R^d)$ the space of continuous, $\bbF$-adapted processes $Y = (Y_t)_{0 \leq t \leq T}$ with 
\begin{align*}
    \norm{Y}^2_{\cS^2} = \E\left[ \sup_{0 \leq t \leq T} |Y_t|^2 \right] < \infty. 
\end{align*}
We also use $L^2 = L^2(\bbF; \R^d)$ to denote the space of square-integrable, $\bbF$-progressive processes $Z$ with 
\begin{align*}
    \norm{Z}^2_{L^2} = \E\bigg[ \int_0^T |Z_t|^2 dt \bigg] < \infty.
\end{align*}
Here we establish that the Pontryagin system \eqref{pontryagin}
has a unique solution $(X,Y,Z,Z^0)$ in the space $\cS^2 \times \cS^2 \times L^2 \times L^2$, and verify that if $m \in \cP_p(\R^d)$ for some $p  >2$, then this bound propagates to both $X$ and $Y$. We begin with a Lemma which shows that $\scrH$ satisfies an appropriate monotonicity condition. 

\begin{lemma} \label{lem.scrH.mon}
    For any square-integrable, $d$-dimensional random vectors $X,\ov X$, $Y, \ov Y$, we have 
    \begin{align} \label{scrH.mon}
      \nonumber   &\E\Big[ \Big( D_p \scrH\big(X,Y,\cL(X,Y)\big) - D_p \scrH\big(\ov X,\ov Y,\cL(\ov X,\ov Y)\big) \Big) \cdot (Y - \ov Y) 
        \\
        &\qquad -\Big( D_x \scrH\big(X,Y,\cL(X,Y)\big) - D_x \scrH\big(\ov X,\ov Y,\cL(\ov X,\ov Y)\big) \Big) \cdot (X - \ov X)  \Big] 
      \nonumber  \\
        &\qquad \qquad  \geq \scrC_{L,a} \E\Big[ \big| D_p \scrH\big(X,Y,\cL(X,Y)\big) - D_p \scrH\big(\ov X,\ov Y,\cL(\ov X,\ov Y)\big) \big|^2 \Big] - \scrC_{L,x} \E\big[ |X - \ov X|^2 \big].
    \end{align}
\end{lemma}
\begin{proof}
    Set
    \begin{align*}
        \alpha = - D_p \scrH\big(X, Y, \cL(X,Y)\big), \quad \ov \alpha = - D_p \scrH\big(\ov X, \ov Y, \cL(\ov X,\ov Y) \big), 
    \end{align*}
    so that $\Phi\big(\cL(X,Y)\big) = \cL(X,\alpha)$. 
    Using the definition of $\scrH$,
    \begin{align*}
        Y = - D_aL \Big(X, - D_pH\big(X,Y,\cL(X,\alpha)\big), \cL(X,\alpha) \Big) = - D_aL\big(X,\alpha, \cL(X,\alpha) \big),
    \end{align*}
    and similarly for $\ov Y$. Meanwhile, 
    \begin{align*}
        D_{x^i} \scrH\big(X,Y,\cL(X,Y)\big) = D_{x^i} H\big( X, Y, \cL(X,\alpha)  \big) = - D_{x^i} L\big(X,\alpha, \cL(X,\alpha) \big),
    \end{align*}
    and likewise for $D_{x^i} \scrH\big(\ov X, \ov Y, \cL(\ov X,\ov Y) \big)$. Thus, we see that \eqref{lem.scrH.mon} is just a rewriting of \eqref{dispmonotone.L}.
\end{proof}

\begin{proposition}
    \label{prop.app.wellposed}
    The Pontryagin system \eqref{pontryagin} has a unique solution in $\cS^2 \times \cS^2 \times L^2 \times L^2$. Moreover, there is a constant $C > 0$ such that if $(X,Y,Z,Z^0)$ satisfies \eqref{pontryagin} and $(\ov X, \ov Y, \ov Z, \ov Z^0)$ solves \begin{align} \label{pontryagin.errors}
    \begin{cases}
       \ds  dX_t = \Big( - D_p \scrH\big(X_t,Y_t,\cL^0(X_t,Y_t) \big) + \cE_t^{1} \Big) dt + \sqrt{2} dW_t + \sqrt{2\sigma_0} dW_t^0, \quad 0 \leq t \leq T, 
       \vspace{.2cm}  \\
        \ds dY_t = \Big( D_x \scrH \big(X_t,Y_t, \cL^0(X_t,Y_t) \big) + \cE_t^{2} \Big) dt + Z_t dW_t + Z_t^0 dW_t^0, \quad 0 \leq t \leq T, \vspace{.2cm} 
        \\
        \ds X_0 = \xi + \cE^0, \quad Y_T = D_x G(X_T,\cL^0(X_T)) + \cE^3, 
    \end{cases}
\end{align}
with $\cE^0$ and $\cE^3$ square-integrable, $\R^d$-valued random vectors measurable with respect to $\cF_0$ and $\cF_T$, respectively, and $\cE^1 = (\cE^1_t)_{0 \leq t \leq T}$ and $\cE^2 = (\cE^2_t)_{0 \leq t \leq T}$ square-integrable, $\bbF$-adapted processes taking values in $\R^d$, then we have the estimate
    \begin{align} \label{stab.appendix}
        \E\bigg[ \sup_{0 \leq t \leq T} \Big( \big|X_t - \ov X_t \big|^2 + \big|Y_t - \ov Y_t\big|^2 \Big) \bigg] \leq C \E\bigg[ \big| \cE^0 \big|^2 + \big|\cE^3 \big|^2 + \int_0^T \Big( \big| \cE_t^1 \big|^2 + \big| \cE_t^2 \big|^2 \Big)dt \bigg]
    \end{align}
\end{proposition} 

\begin{proof}
    We begin by assuming that $(X,Y,Z,Z^0)$ and $(\ov X, \ov Y, \ov Z, \ov Z^0)$ satisfy \eqref{pontryagin} and \eqref{pontryagin.errors}, and we prove the a-priori stability estimate \eqref{stab.appendix}. In light of Lemma \ref{lem.scrH.mon}, we can follow a standard argument for monotone FBSDEs. Suppose that $(X,Y,Z,Z^0)$ and $(\ov X, \ov Y, \ov Z, \ov Z^0)$ are solutions. Set 
    \begin{align*}
        \Delta X = X - \ov X, \quad \Delta Y = Y - \ov Y, \quad \Delta Z = Z - \ov Z, \quad \Delta Z^0 = Z - \ov Z. 
    \end{align*}
    Also set 
    \begin{align*}
        \alpha_t = - D_p \scrH\big(X_t,Y_t,\cL^0(X_t,Y_t)\big), \quad \ov \alpha_t = - D_p\scrH\big(\ov X_t, \ov Y_t, \cL^0(\ov X_t,\ov Y_t)\big), \quad \Delta \alpha_t = \alpha_t - \ov \alpha_t.
    \end{align*}
    Compute 
    \begin{align*}
        d\big( \Delta X_t \cdot \Delta Y_t \big) &= \bigg( \Delta X_t \cdot \Big( D_x \scrH\big(X_t,Y_t,\cL^0(X_t,Y_t)\big) - D_x \scrH\big(\ov X_t,\ov Y_t,\cL^0(\ov X_t,\ov Y_t)\big)  - \cE_t^2 \Big) 
        \\
        &\qquad \qquad - \Delta Y_t \cdot  \Big( D_p \scrH\big(X_t,Y_t,\cL^0(X_t,Y_t)\big) - D_p \scrH\big(\ov X_t,\ov Y_t,\cL^0(\ov X_t,\ov Y_t)\big)  + \cE_t^1 \Big)  + dM_t, 
    \end{align*}
    with $M$ being a martingale whose form is not important. Integrate from $0$ to $T$, apply Lemma \ref{lem.scrH.mon}, use the fact that $\Delta X_0 = \cE^0$, and take expectations to get 
    \begin{align} \label{moncomp1}
      \nonumber  \scrC_{L,a} \E\bigg[ \int_0^T \big| \Delta \alpha_t\big|^2 dt \bigg] &\leq  \E\bigg[  \int_0^T \Big( \scrC_{L,x}\big| \Delta X_t \big|^2 + \big|\Delta Y_t\big| \big|\cE_t^1 \big| + \big| \Delta X_t \big| \big| \cE_t^2 \big| \Big) dt - \Delta X_T \cdot \Delta D_x G - \cE^3 \cdot \Delta D_x G + \Delta Y_0 \cdot \cE^0 \bigg] 
       \\
       &\leq \E\bigg[  \int_0^T \Big( \scrC_{L,x}\big| \Delta X_t \big|^2 + \big|\Delta Y_t\big| \big|\cE_t^1 \big| + \big| \Delta X_t \big| \big| \cE_t^2 \big| \Big) dt + \scrC_{G} \big|\Delta X_T\big|^2 + \big| \Delta D_x G \big| \big|\cE^3\big| + \big| \Delta Y_0 \big| \big| \cE^0 \big| \bigg],
    \end{align}
    where $\Delta D_x G = 
    \Big( D_x G\big(X_T,\cL^0(X_T)\big) -  D_x G\big(\ov X_T,\cL^0(\ov X_T)\big) \Big)$. 
    Next, use the fact that $\Delta X_t = \Delta \xi + \int_0^t \Delta \alpha ds$ to find 
    \begin{align}
        \E\big[ |\Delta X_t|^2 \big] \leq C_{\eta} \E\big[ \big|\cE^0\big|^2 \big] + (1 + \eta)t \E\Big[ \int_0^T |\Delta \alpha_s|^2 ds \Big]
    \end{align} 
    for any $\eta > 0$,
    and plug this into \eqref{moncomp1} to obtain
    \begin{align} \label{moncomp2}
        \big(\scrC_{\dis} - \eta\big) \E\bigg[ \int_0^T \big| \Delta \alpha_t\big|^2 dt \bigg] &\leq  \E\bigg[  \int_0^T \Big( \big|\Delta Y_t\big| \big|\cE_t^1 \big| + \big| \Delta X_t \big| \big| \cE_t^2 \big| \Big) dt + \big| \Delta D_x G \big| \big|\cE^3\big| + \big| \Delta Y_0 \big| \big| \cE^0 \big| \bigg]
    \end{align}
    We choose $\eta = \scrC_{\dis}/2$. 
    By expanding $t \mapsto \E\big[ |\Delta Y_t|^2 \big]$ and using Gr\"onwall, we find that there is a constant $C$ such that 
    \begin{align*}
        \E\Big[ \sup_{0 \leq t \leq T} & |\Delta Y_t|^2 \Big] \leq C \E\Big[\sup_{0 \leq t \leq T} |\Delta X_t|^2 \Big] 
       \leq C \E\bigg[ \int_0^T \big| \Delta \alpha_t\big|^2 dt \bigg] + C  \E\big[ \big|\Delta Y_0 \big| \big| \cE^0 \big| \big], 
    \end{align*}
   and hence, coming back to \eqref{moncomp2}, and then applying Young's inequality, we find that  
   \begin{align} \label{moncomp4}
        \big(\scrC_{\dis} - \eta\big) \E\bigg[ \int_0^T \big| \Delta \alpha_t\big|^2 dt \bigg] &\leq  \int_0^T \Big( \big| \cE_t^1 \big|^2 + \big| \cE_t^2 \big|^2 \Big)dt \bigg].
    \end{align}
    The corresponding bounds on $\Delta X$ and then $\Delta Y$ follow. 
    
    With the stability estimate in hand (with constants which depend only on the regularity and monotonicity of the coefficients $D_p \scrH$, $D_x \scrH$, and $D_x G$), a straightforward adaptation of the method of continuation yields existence, see e.g. the proof of Theorem A.2 in \cite{JacTan} for an execution of this argument in the McKean--Vlasov setting. We omit the details.
\end{proof}

We now address the integrability of solutions. For this, it will be useful to use the fact that the FBSDE \eqref{pontryagin} admits a decoupling field. More precisely, we define, for each $(t_0,x_0,m_0)$, 
\begin{align}\label{Psidef}
    \Psi(t_0,x_0,m_0) = Y_{t_0}^{t_0,\xi_0,x_0}, \quad \text{with } \xi_0 \sim m_0, 
\end{align}
where $\big(X^{t_0,\xi_0}, Y^{t_0,\xi_0}, Z^{t_0,\xi_0}, Z^{t_0,\xi_0,0}\big)$ denotes the solution of 
\begin{align} \label{pontryagin.diff}
    \begin{cases}
       \ds  dX^{t_0,\xi_0}_t = - D_p \scrH\Big(X^{t_0,\xi_0}_t,Y^{t_0,\xi_0}_t,\nu_t^{t_0,\xi_0} \Big) dt + \sqrt{2} dW_t + \sqrt{2\sigma_0} dW_t^0, \quad t_0 \leq t \leq T, 
       \vspace{.2cm}  \\
        \ds dY^{t_0,\xi_0}_t = D_x \scrH \Big(X^{t_0,\xi_0}_t,Y^{t_0,\xi_0}_t, \nu_t^{t_0,\xi_0} \Big) dt + Z^{t_0,\xi_0}_t dW_t + Z^{t_0,\xi_0,0}_t dW_t^0, \quad t_0 \leq t \leq T, \vspace{.2cm} 
        \\
        \ds X_{t_0}^{t_0,\xi_0} = \xi_0, \quad Y^{t_0,\xi_0}_T = D_x G\big(X^{t_0,\xi_0}_T, (\nu_T^{t_0,\xi_0})^x \big), \quad \nu_t^{t_0,\xi_0} = \cL^0\big(X_t^{t_0,\xi_0}, Y_t^{t_0,\xi_0}\big),
    \end{cases}
\end{align}
while $\big(X^{t_0,\xi_0,x_0}, Y^{t_0,\xi_0, x_0}, Z^{t_0,\xi_0, x_0}, Z^{t_0,\xi_0,x_0, 0}\big)$ solve
\begin{align} \label{pontryagin.diff2}
    \begin{cases}
       \ds  dX^{t_0,\xi_0, x_0}_t = - D_p \scrH\Big(X^{t_0,\xi_0, x_0}_t,Y^{t_0,\xi_0, x_0}_t, \nu_t^{t_0,\xi_0} \Big) dt + \sqrt{2} dW_t + \sqrt{2\sigma_0} dW_t^0, \quad t_0 \leq t \leq T, 
       \vspace{.2cm}  \\
        \ds dY^{t_0,\xi_0, x_0}_t = D_x \scrH \Big(X^{t_0,\xi_0, x_0}_t,Y^{t_0,\xi_0, x_0}_t, \nu_t^{t_0,\xi_0} \Big) dt + Z^{t_0,\xi_0,x_0}_t dW_t + Z^{t_0,\xi_0,x_0,0}_t dW_t^0, \quad t_0 \leq t \leq T, \vspace{.2cm} 
        \\
        \ds X_{t_0}^{t_0,\xi_0,x_0} = x_0, \quad Y^{t_0,\xi_0,x_0}_T = D_x G\big(X^{t_0,\xi_0,x_0}_T, (\nu_T^{t_0,\xi_0})^x\big).
    \end{cases}
\end{align}
We note that the well-posedness and stability result in Proposition \ref{prop.app.wellposed} carries over easily to \eqref{pontryagin.diff}. Then \eqref{pontryagin.diff2} becomes a finite-dimensional monotone FBSDE with random coefficients; its well-posedness can be established by following the argument in \cite[Lemma 2.2]{MesMou} to show that for each fixed $\nu \in \cP_2(\R^d \times \R^d)$, we have
\begin{align*}
     - \big(D_x \scrH(x,p,\nu)& - D_x\scrH(\ov x, \ov p, \ov \nu) \cdot(x - \ov x) + \big(D_p \scrH(x,p,\nu) - D_p\scrH(\ov x, \ov p, \ov \nu) \cdot(p - \ov p) 
     \\
     &\geq \scrC_{L,a} \big|D_x \scrH(x,p,\nu) - D_x\scrH(\ov x, \ov p, \ov \nu)\big|^2 - \scrC_{L,x} |x - \ov x|^2, 
\end{align*}
for all $x,\ov x, p, \ov p \in \R^d$, as well as
\begin{align*}
   \big( D_x G(x,m) - D_x G(\ov x, m) \big) \cdot (x - \ov x) \geq - \scrC_G |x - \ov x|^2
\end{align*}
for each $m \in \cP_2(\R^d)$ and $x,\ov x \in \R^d$. Thus, \eqref{pontryagin.diff2} is a finite-dimensional, monotone FBSDE whose well-posedness can be obtained via the method of continuation as in \cite{hupengmonotone} or \cite{PengWumonotone}.

Following e.g. \cite[Section 5.1,2,]{Carmona2018}, we have that $\Psi$ is well-defined and we have 
\begin{align*}
    Y^{t_0,\xi_0}_t = \Psi\big(t,X_t^{t_0,\xi_0} ,\cL^0(X_t^{t_0,\xi_0}) \big).
\end{align*}

\begin{lemma} \label{lem.psilingrowth}
    The map $\Psi$ satisfies 
    \begin{align*}
        |\Psi(t,x,m)| \leq C\Big(1 + |x| + \big(M_2(m)\big)^{1/2} \Big).
    \end{align*}
\end{lemma}
\begin{proof}
    This is a consequence of the stability estimate in Proposition \ref{prop.app.wellposed}.
\end{proof}
\begin{lemma} \label{lem.Lpbound}
    Let $(X,Y,Z)$ be the unique solution to \eqref{pontryagin} with initial data $\xi \in L^p(\cF_0)$, $p > 2$. Then we have 
    \begin{align*}
        \E\left[ \sup_{0 \leq t \leq T} |X_t|^p + \sup_{0 \leq t \leq T} |Y_t|^p \right] < +\infty.
    \end{align*}
\end{lemma}

\begin{proof}
    We have 
    \begin{align*}
        dX_t = - D_p\scrH\Big( X_t, \Psi(t, X_t, \cL^0(X_t)), \cL^0\big( X_t, \Psi(t, X_t, \cL^0(X_t)) \big)\Big) dt + \sqrt{2} dW_t + \sqrt{2\sigma_0} dW_t^0, 
    \end{align*}
    where $\Psi$ is defined by \eqref{Psidef}.
   For simplicity, we set 
   \begin{align*}
       b_t = - D_p\scrH\Big( X_t, \Psi(t, X_t, \cL^0(X_t)), \cL^0\big( X_t, \Psi(t, X_t, \cL^0(X_t)) \big)\Big). 
   \end{align*}
The Lipschitz regularity of $- D_p \scrH$ and the linear growth of $\Psi$ (from Lemma \ref{lem.psilingrowth} yield 
    \begin{align*}
        |b_t| \leq C \Big(1 + |X_t| + \E\Big[ |X_t|^2 | \cF_t^0 \Big]^{1/2} \Big). 
    \end{align*}
    We now compute 
    \begin{align*}
        d|X_t|^p  = \Big( p|X_t|^{p-2} X_t \cdot b_t + p(p + d - 2)(1 + \sigma_0) |X_t|^{p-2} \Big) dt + dM_t, 
    \end{align*}
    with $M$ being a local martingale. We will argue as if $M$ is a true martingale, this assumption being easily removed by a localization argument. We thus have 
    \begin{align*}
        \frac{d}{dt} &\E\big[ |X_t|^p \big] \leq C\Big(1 +  \E\Big[ |X_t|^{p-1} \Big( |X_t| + \E\Big[|X_t|^2 | \cF_t^0 \Big]^{1/2} \Big)  \Big]
        \\
        &\qquad \leq C \Big(1 + \E\big[ |X_t|^p \big] + \E\Big[ \E\big[ |X_t|^2  | \cF_t^0 \big]^{p/2} \Big] \Big) \leq C \Big( 1 + \E\Big[ |X_t|^p \Big] \Big). 
    \end{align*}
    An application of Gr\"onwall's Lemma completes the proof gives the desired estimate for a fixed time, and then the uniform in time results follow easily.
\end{proof}

\bibliographystyle{alpha}
\bibliography{joe_alpar}

\end{document}